\documentclass[a4paper, reqno]{amsart}

\usepackage[UKenglish]{babel}
\usepackage{amsmath, amssymb, amsthm}
\usepackage{scrextend}
\usepackage[hidelinks]{hyperref}
\usepackage{xpatch}
\usepackage{color}
\usepackage{tikz-cd}
\usepackage{enumitem}
\usepackage{theoremref}

\usepackage[inline,final]{showlabels}
\showlabels{thlabel}

\setlist[description]{font=\normalfont\scshape}

\xpatchcmd{\proof}{\itshape}{\normalfont\bfseries}{}{}
\newtheoremstyle{repeat}{}{}{\itshape}{}{\bfseries}{.}{.5em}{#3, repeated}

\newtheorem{theorem}{Theorem}[section]
\newtheorem{proposition}[theorem]{Proposition}
\newtheorem{lemma}[theorem]{Lemma}
\newtheorem{corollary}[theorem]{Corollary}
\newtheorem{fact}[theorem]{Fact}

\theoremstyle{definition}
\newtheorem{definition}[theorem]{Definition}
\newtheorem{remark}[theorem]{Remark}
\newtheorem{convention}[theorem]{Convention}
\newtheorem{example}[theorem]{Example}

\theoremstyle{repeat}
\newtheorem*{repeated-theorem}{Repeat}

\newcommand{\B}{\mathcal{B}}
\newcommand{\C}{\mathcal{C}}
\newcommand{\F}{\mathcal{F}}
\newcommand{\K}{\mathcal{K}}
\newcommand{\M}{\mathcal{M}}

\newcommand{\Mod}{\mathbf{Mod}}
\newcommand{\SubMod}{\mathbf{SubMod}}
\newcommand{\MetMod}{\mathbf{MetMod}}
\newcommand{\SubMetMod}{\mathbf{SubMetMod}}
\newcommand{\SubSet}{\mathbf{SubSet}}

\DeclareMathOperator{\colim}{colim}
\DeclareMathOperator{\tp}{tp}
\DeclareMathOperator{\gtp}{gtp}
\DeclareMathOperator{\Lgtp}{Lgtp}
\DeclareMathOperator{\dom}{dom}
\DeclareMathOperator{\Sub}{Sub}
\DeclareMathOperator{\base}{base}
\DeclareMathOperator{\Hom}{Hom}
\DeclareMathOperator{\Sgtp}{S_{gtp}}

\newcommand{\ld}{\textup{ld}}
\newcommand{\isid}{\textup{isi-d}}
\newcommand{\isif}{\textup{isi-f}}
\newcommand{\lK}{\textup{lK}}

\renewcommand{\phi}{\varphi}

\def\Ind#1#2{#1\setbox0=\hbox{$#1x$}\kern\wd0\hbox to 0pt{\hss$#1\mid$\hss}
\lower.9\ht0\hbox to 0pt{\hss$#1\smile$\hss}\kern\wd0}
\def\ind{\mathop{\mathpalette\Ind{}}}
\def\Notind#1#2{#1\setbox0=\hbox{$#1x$}\kern\wd0\hbox to 0pt{\mathchardef
\nn="3236\hss$#1\nn$\kern1.4\wd0\hss}\hbox to 0pt{\hss$#1\mid$\hss}\lower.9\ht0
\hbox to 0pt{\hss$#1\smile$\hss}\kern\wd0}

\title{NSOP$_1$-like independence in AECats}
\author{Mark Kamsma}
\email[Mark Kamsma]{mark@markkamsma.nl}
\urladdr{https://markkamsma.nl}
\date{\today. \emph{MSC2020}: Primary: 03C45; secondary: 03C48, 03C52, 03C66, 18C35}
 
\begin{document}

\begin{abstract}
The classes stable, simple and NSOP$_1$ in the stability hierarchy for first-order theories can be characterised by the existence of a certain independence relation. For each of them there is a canonicity theorem: there can be at most one nice independence relation. Independence in stable and simple first-order theories must come from forking and dividing (which then coincide), and for NSOP$_1$ theories it must come from Kim-dividing.

We generalise this work to the framework of AECats (Abstract Elementary Categories) with the amalgamation property. These are a certain kind of accessible category generalising the category of (subsets of) models of some theory. We prove canonicity theorems for stable, simple and NSOP$_1$-like independence relations. The stable and simple cases have been done before in slightly different setups, but we provide them here as well so that we can recover part of the original stability hierarchy. We also provide abstract definitions for each of these independence relations as what we call isi-dividing, isi-forking and long Kim-dividing.
\end{abstract}

\maketitle

\tableofcontents

\section{Introduction}
\label{sec:introduction}
Independence relations are a central notion in model theory. Work on independence in first-order theories was started by Shelah \cite{shelah_classification_1990} through the notions of forking and dividing. This was later generalised to simple theories in work by Kim and Pillay \cite{kim_forking_1998, kim_simple_1997}. In NSOP$_1$ theories dividing is no longer so well-behaved in general. Inspired by ideas from Kim \cite{kim_ntp1_2009}, Kaplan and Ramsey developed the notion of Kim-dividing \cite{kaplan_kim-independence_2020}, which does yield a nice independence relation in NSOP$_1$ theories. Each of these classes admits a so-called \emph{Kim-Pillay style theorem}, after a result by Kim and Pillay \cite[Theorem 4.2]{kim_simple_1997}. Roughly the statement is as follows:
\begin{quote}
A theory is simple if and only if it admits an independence relation satisfying a certain list of properties. Furthermore, in this case that relation is given by forking independence.
\end{quote}
In particular such a theorem gives us \emph{canonicity}: there can be at most one nice enough independence relation, which must be forking independence.

All of the above takes place in the classical framework of first-order logic. However, there are many interesting classes of structures that do not fit in this framework. Similar work has been done in different and more general logical frameworks. For example, the stable and simple settings have been studied in positive logic \cite{shelah_lazy_1975,pillay_forking_2000, ben-yaacov_simplicity_2003}, continuous logic \cite{ben-yaacov_model_2008} and AECs \cite{shelah_classification_2009, hyttinen_independence_2006, boney_canonical_2016, vasey_building_2016, grossberg_simple-like_2021}. More recently the NSOP$_1$ setting has been studied in positive logic \cite{haykazyan_existentially_2021, dobrowolski_kim-independence_2022}. It also makes an appearance in continuous logic in \cite{berenstein_hilbert_2018}, where a non-simple NSOP$_1$ continuous theory is studied. Even then there is a more general category-theoretic approach, unifying all the previously mentioned frameworks. Lieberman, Rosick\'y and Vasey proved a category-theoretic canonicity theorem for stable independence relations \cite{lieberman_forking_2019}. In \cite{kamsma_kim-pillay_2020} a category-theoretic canonicity theorem for simple independence relations was proved. In this paper we continue this work and prove a canonicity theorem for NSOP$_1$-like independence relations.

We work in the same framework as in \cite{kamsma_kim-pillay_2020}, namely the framework of AECats (Abstract Elementary Categories) with the amalgamation property. This generalises both the category of models of some first-order theory $T$ and the category of subsets of models of $T$. The framework can also be applied to positive logic, continuous logic and AECs (\thref{ex:aecats}).

An independence relation will be defined as in \cite{kamsma_kim-pillay_2020} and will be a relation on triples of subobjects, where we use the notation $A \ind_C^M B$ to say that subobjects $A, B, C$ of $M$ are independent. However, it will be useful to restrict the objects that can appear in the base of the independence relation (i.e.\ the $C$ in $A \ind_C^M B$). For example, we might only want to consider independence over models while allowing arbitrary sets on the sides. We add this flexibility in this paper through the notion of a \emph{base class}, which will be the class of objects that is allowed in the base.

An independence relation will be called stable, simple or NSOP$_1$-like based on the properties that it satisfies (\thref{def:independence-relation-types}). These are the appropriate category-theoretic versions of the properties that we know independence to have in the corresponding classes in the classical first-order setting. In particular this means that any stable independence relation is simple, and any simple independence relation is NSOP$_1$-like, reflecting (that part of) the original stability hierarchy.

In \cite{kamsma_kim-pillay_2020} the notion of isi-dividing was introduced, and the main result stated that any simple independence relation comes from isi-dividing (i.e.\ any simple independence relation is non-isi-dividing). As discussed before, in first-order NSOP$_1$ theories the notion of dividing is no longer so well-behaved, and we should study Kim-dividing instead. So if we think of isi-dividing as the analogue of dividing in AECats, then we will need the right analogue of Kim-dividing to study NSOP$_1$-like independence relations in AECats. To this end we define long Kim-dividing (\thref{def:long-kim-dividing}). In this definition we need a forking notion based on isi-dividing, which we will then call isi-forking (\thref{def:isi-forking}).

Before we move on to the main results we make a quick comment about some terminology. Classically being stable or simple is defined as not having the order property (NOP) or not having the tree property (NTP) respectively. In line with this, NSOP$_1$ means that we do not have SOP$_1$, but there is no separate name for NSOP$_1$. We do not consider the combinatorial properties OP, TP and SOP$_1$ in this paper. It is not even directly clear what these should look like in settings without compactness. A link between stable independence relations and a form of the order property is established in \cite{lieberman_forking_2019} in a category-theoretic setting. There is also \cite{grossberg_simple-like_2021}, which studies the connection between various tree properties and simple independence relations in AECs. For NSOP$_1$ there is the work \cite{haykazyan_existentially_2021, dobrowolski_kim-independence_2022} that does consider the combinatorial property SOP$_1$ in positive logic. However, this is not nearly as general as the category-theoretic setting that we aim for here. This is why we use the term ``NSOP$_1$-like independence relation'', because it is an independence relation that is classically found in NSOP$_1$ theories, even though we do not consider the relevant combinatorial property.

\textbf{Main results}. Our main result is canonicity of NSOP$_1$-like independence relations. In the statement below $\ind^\lK$ denotes the independence relation obtained from long Kim-dividing.
\begin{theorem}[Canonicity of NSOP$_1$-like independence]
\thlabel{thm:canonicity-of-nsop1-like-independence}
Let $(\C, \M)$ be an AECat with the amalgamation property and let $\B$ be some base class. Suppose that $(\C, \M)$ satisfies the $\B$-existence axiom and suppose that there is an NSOP$_1$-like independence relation $\ind$ over $\B$. Then $\ind = \ind^{\lK}$ over $\B$.
\end{theorem}
For a discussion about the assumption of the $\B$-existence axiom we refer to \thref{ex:b-existence}. All we say now is that it is a reasonable, and necessary, assumption, already in the very concrete setting of first-order logic.

We also slightly improve the main result from \cite{kamsma_kim-pillay_2020} on canonicity of simple independence relations. In the statement below $\ind^{\isid}$ and $\ind^{\isif}$ denote the independence relations obtained from isi-dividing and isi-forking respectively and $\base(\ind)$ denotes the class of objects that are allowed in the base of $\ind$. The slight improvement over \cite{kamsma_kim-pillay_2020} is the fact that we can restrict the base of our independence relation and the fact that we also get that $\ind = \ind^{\isif}$.
\begin{theorem}[Canonicity of simple independence]
\thlabel{thm:canonicity-of-simple-independence}
Let $(\C, \M)$ be an AECat with the amalgamation property, and suppose that $\ind$ is a simple independence relation. Then $\ind = \ind^{\isid} = \ind^{\isif}$ over $\base(\ind)$.
\end{theorem}
Combining the two main theorems we can compare stable, simple and NSOP$_1$-like independence relations, even without assuming the $\B$-existence axiom. This allows us to recover part of the original stability hierarchy based on independence relations, see also \thref{rem:recover-part-of-stability-hierarchy}.
\begin{theorem}
\thlabel{thm:independence-hierarchy}
Let $(\C, \M)$ be an AECat with the amalgamation property and suppose that $\ind$ is a stable or a simple independence relation in $(\C, \M)$. Suppose furthermore that $\ind^*$ is an NSOP$_1$-like independence relation in $(\C, \M)$ with $\base(\ind) = \base(\ind^*)$. Then
\[
\ind = \ind^*  = \ind^{\isid} = \ind^{\isif} = \ind^{\lK}.
\]
\end{theorem}
\textbf{Overview.} We start by recalling the framework of AECats in Section \ref{sec:aecats}. To make sense of types in this framework we also recall the notion of Galois types.

In Section \ref{sec:lascar-strong-galois-types} we will define what we call \emph{Lascar strong Galois types}, based on the notion of Lascar strong types in first-order logic. These are necessary for a property called \textsc{Independence Theorem} for independence relations later.

We recall the notion of an independence relation in an AECat in Section \ref{sec:independence-relations}. We also recall the notion of an \emph{independent sequence} in this section, and prove that these exist assuming only very few basic properties for an independence relation.

In Section \ref{sec:independence-theorem-3-amalgamation-and-stationarity} we investigate the well-known equivalence between the properties \textsc{Independence Theorem}  and \textsc{3-amalgamation} and we prove this fact for AECats. We also recall that \textsc{Stationarity} implies both of them. All of this is over models.

In Section \ref{sec:long-dividing} we recall the notion of isi-dividing and introduce the notions of long dividing, isi-forking and long Kim-dividing. We also discuss connections to the classical analogues: dividing, forking and Kim-dividing.

Section \ref{sec:canonicity} contains the main results, the canonicity theorems. After those theorems we discuss how this work extends and brings together previously known results in different (less general) frameworks.

Finally, Section \ref{sec:lgtp-vs-bounded-gtp} explores the notion of Lascar strong Galois type further. Lascar strong types are known to heavily interact with independence relations in first-order logic and we prove that this is also the case for Lascar strong Galois types.

\textbf{Acknowledgements.} I would like to thank Jonathan Kirby for his feedback which greatly improved the presentation of this paper. I would also like to thank the anonymous referee whose remarks further helped to improve the presentation of this paper. This paper is part of a PhD project at the UEA (University of East Anglia), and as such is supported by a scholarship from the UEA.
\section{AECats}
\label{sec:aecats}
In this section we recall definitions and basic results for AECats from \cite{kamsma_kim-pillay_2020}. We assume that the reader is familiar with the framework of accessible categories. A great reference for this is \cite{adamek_locally_1994}.
\begin{convention}
\thlabel{conv:regular-cardinals}
Throughout this paper we are only interested in regular cardinals, which we usually denote by $\kappa$, $\lambda$ and $\mu$.
\end{convention}
\begin{definition}[{\cite[Definition 2.5]{kamsma_kim-pillay_2020}}]
\thlabel{def:aecat}
An \emph{AECat}, short for \emph{abstract elementary category}, consists of a pair $(\C, \M)$ where $\C$ and $\M$ are accessible categories and $\M$ is a full subcategory of $\C$ such that:
\begin{enumerate}[label=(\roman*)]
\item $\M$ has directed colimits, which the inclusion functor into $\C$ preserves;
\item all arrows in $\C$ (and thus in $\M$) are monomorphisms.
\end{enumerate}
The objects in $\M$ are called \emph{models}. We say that $(\C, \M)$ has the \emph{amalgamation property} (or \emph{AP}) if $\M$ has the amalgamation property.
\end{definition}
We refer to \cite{kamsma_kim-pillay_2020} for the motivation of this definition and elaborate examples. Below we just summarise some examples of AECats.
\begin{example}
\thlabel{ex:aecats}
The following are examples of AECats.
\begin{enumerate}[label=(\roman*)]
\item For a first-order theory $T$ we write $\Mod(T)$ for the category of models with elementary embeddings. Then $(\Mod(T), \Mod(T))$ is an AECat with AP.
\item Fix some first-order theory $T$. Write $\SubMod(T)$ for the category of subsets of models of $T$. That is, objects are pairs $(A, M)$ where $A \subseteq M$ and $M \models T$. An arrow $f: (A, M) \to (B, N)$ is an elementary map $f: A \to B$. The inclusion $\Mod(T) \hookrightarrow \SubMod(T)$ sending $M$ to $(M, M)$ is full and faithful. Thus $(\SubMod(T), \Mod(T))$ is an AECat with AP.
\item Examples (i) and (ii) generalise to positive logic, because any first-order theory can be seen as a positive theory through Morleyisation. So we use the same notation. That is, for a positive theory $T$ we have the category $\Mod(T)$ of existentially closed models and immersions and the category $\SubMod(T)$ of subsets of existentially closed models with immersions between those subsets. Then we have the following AECats with AP: $(\Mod(T), \Mod(T))$ and $(\SubMod(T), \Mod(T))$.
\item A similar construction to examples (i) and (ii) is possible for continuous logic. For a continuous theory $T$ we can form categories $\MetMod(T)$ of models of $T$ and $\SubMetMod(T)$ of closed subsets of models of $T$, see \cite[Example 2.10]{kamsma_kim-pillay_2020} for more details. Then $(\MetMod(T), \MetMod(T))$ and $(\SubMetMod(T), \MetMod(T))$ are AECats with AP.
\item Let $\K$ be an AEC. We view $\K$ as a category by taking as arrows $\K$-embeddings. Then $(\K, \K)$ is an AECat and has AP iff $\K$ has AP.
\end{enumerate}
\end{example}
In \cite[Example 2.11]{kamsma_kim-pillay_2020} there is also a construction to consider subsets of the structures in an AEC, similar to the construction of $\SubMod(T)$. There is a simpler construction that does not need the assumptions on $\K$ in that example.
\begin{example}
\thlabel{ex:subsets-aec}
Let $\K$ be an AEC. We define the category of subsets of $\K$, written as $\SubSet(K)$, as follows. Objects are pairs $(A, M)$ where $A \subseteq M$ and $M \in \K$. An arrow $f: (A, M) \to (B, N)$ is then a $\K$-embedding $f: M \to N$ such that $f(A) \subseteq B$. One easily verifies that $(\SubSet(\K), \K)$ is an AECat, and it has AP exactly when $\K$ has AP.
\end{example}
\begin{definition}
\thlabel{def:kappa-aecat}
We call an AECat $(\C, \M)$ a \emph{$\kappa$-AECat} if $\C$ and $\M$ are both $\kappa$-accessible and the inclusion functor preserves $\kappa$-presentable objects.
\end{definition}
\begin{fact}[{\cite[Remark 2.8]{kamsma_kim-pillay_2020}}]
\thlabel{fact:kappa-aecat}
For any AECat $(\C, \M)$ there are arbitrarily large $\kappa$ such that $(\C, \M)$ is a $\kappa$-AECat.
\end{fact}
\begin{proposition}
\thlabel{prop:kappa-aecat-and-above}
Let $(\C, \M)$ be a $\kappa$-AECat and let $\lambda \geq \kappa$. Then $\M$ is $\lambda$-accessible and the inclusion functor preserves and reflects $\lambda$-presentable objects.
\end{proposition}
\begin{proof}
The claim about $\lambda$-accessibility is exactly \cite[Proposition 4.1]{beke_abstract_2012}. Preservation and reflection of $\lambda$-presentable objects follows using the same proofs as \cite[Proposition 4.3]{beke_abstract_2012} and \cite[Lemma 3.6]{beke_abstract_2012} respectively, where in the latter we use the former and the fact that all arrows in $\C$ are monomorphisms, so the inclusion functor reflects split epimorphisms (which are isomorphisms).
\end{proof}
\begin{definition}
\thlabel{def:extension}
Let $M$ be a model in an AECat. An \emph{extension} of $M$ is an arrow $M \to N$, where $N$ is some model.
\end{definition}
\begin{convention}
\thlabel{conv:extending-monomorphisms}
Usually, there will be only one relevant extension of models. So to prevent cluttering of notation we will not give such an extension a name. Given such an extension $M \to N$ and some arrow $a: A \to M$ we will then denote the arrow $A \xrightarrow{a} M \to N$ by $a$ as well.
\end{convention}
\begin{definition}
\thlabel{def:galois-type}
Let $(\C, \M)$ be an AECat with AP. We will use the notation $((a_i)_{i \in I}; M)$ to mean that the $a_i$ are arrows into $M$ and that $M$ is a model.

We say that two tuples $((a_i)_{i \in I}; M)$ and $((a_i')_{i \in I}; M')$ have the same \emph{Galois type}, and write
\[
\gtp((a_i)_{i \in I}; M) = \gtp((a_i')_{i \in I}; M'),
\]
if $\dom(a_i) = \dom(a_i')$ for all $i \in I$, and there is a common extension $M \to N \leftarrow M'$ such that the following commutes for all $i \in I$:
\[
\begin{tikzcd}[row sep=tiny]
 & N &  \\
M \arrow[ru] &  & M' \arrow[lu] \\
 & A_i \arrow[lu, "a_i"] \arrow[ru, "a_i'"'] & 
\end{tikzcd}
\]
\end{definition}
Note that AP ensures that having the same Galois type is an equivalence relation. For this reason, we are only interested in AECats with AP in the rest of this paper.
\begin{fact}[{\cite[Proposition 3.8]{kamsma_kim-pillay_2020}}]
\thlabel{fact:galois-types}
If $\gtp((a_i)_{i \in I}; M) = \gtp((a_i')_{i \in I}; M')$ then:
\begin{enumerate}[label=(\roman*)]
\item (restriction) we have $\gtp((a_i)_{i \in I_0}; M) = \gtp((a_i')_{i \in I_0}; M')$ for any $I_0 \subseteq I$;
\item (monotonicity) given an arrow $b_i: B_i \to \dom(a_i)$ for each $i \in I$, then
\[
\gtp((a_i)_{i \in I}, (a_i b_i)_{i \in I}; M) = \gtp((a_i')_{i \in I}, (a_i' b_i)_{i \in I}; M')
\]
and thus $\gtp((a_i b_i)_{i \in I}; M) = \gtp((a_i' b_i)_{i \in I}; M')$;
\item (extension) for any $(b; M)$ there is an extension $M' \to N$ and some $(b'; N)$ such that $\gtp(b, (a_i)_{i \in I}; M) = \gtp(b', (a_i')_{i \in I}; N)$.
\end{enumerate}
\end{fact}
\begin{fact}[{\cite[Proposition 3.9]{kamsma_kim-pillay_2020}}]
\thlabel{fact:galois-type-factorisation}
If $\gtp(a, b; M) = \gtp(a', b'; M')$ and $a$ factors through $b$, say as $a = bi$, then $a'$ factors through $b'$ in the same way, so as $a' = b' i$.
\end{fact}
\begin{definition}
\thlabel{def:galois-type-set}
Let $(\C, \M)$ be an AECat with AP. For a tuple $(A_i)_{i \in I}$ of objects in $\C$, let $S((A_i)_{i \in I})$ be the collection of all tuples $((a_i)_{i \in I}; M)$ such that $\dom(a_i) = A_i$. We define the \emph{Galois type set} $\Sgtp((A_i)_{i \in I})$ as:
\[
\Sgtp((A_i)_{i \in I}) = S((A_i)_{i \in I}) / \sim_{gtp},
\]
where $\sim_{gtp}$ is the equivalence relation of having the same Galois type.
\end{definition}
\begin{fact}[{\cite[Proposition 4.6]{kamsma_kim-pillay_2020}}]
\thlabel{fact:galois-type-set}
$\Sgtp((A_i)_{i \in I})$ is really just a set.
\end{fact}
\begin{definition}
\thlabel{def:sequences}
Fix some AECat $(\C, \M)$ with AP.
\begin{enumerate}[label=(\roman*)]
\item A \emph{sequence} is a tuple $((a_i)_{i \in I}; M)$ where every $a_i$ has the same domain and $I$ is a linear order.
\item A \emph{chain} is a diagram of ordinal shape. We call a chain $(M_i)_{i < \kappa}$ \emph{continuous} if $M_\ell = \colim_{i < \ell} M_i$ for all limit $\ell < \kappa$. Given a chain $(M_i)_{i < \kappa}$ we say that $M$ is a \emph{chain bound} for $(M_i)_{i < \kappa}$ if there are arrows $m_i: M_i \to M$ forming a cocone for $(M_i)_{i < \kappa}$.
\item A \emph{chain of initial segments} for some sequence $((a_i)_{i < \kappa}; M)$ is a continuous chain $(M_i)_{i < \kappa}$ of models with chain bound $M$ such that $a_i$ factors through $M_{i+1}$ for all $i < \kappa$.
\item Let $(M_i)_{i < \kappa}$ be a chain with chain bound $M$ and let $c: C \to M$ be some arrow. We say that \emph{$c$ embeds in $(M_i)_{i < \kappa}$} if $c$ factors as $C \to M_0 \to M$.
\item We call a sequence $((a_i)_{i < \kappa}; M)$ together with a chain of initial segments $(M_i)_{i < \kappa}$ an \emph{isi-sequence} (short for \emph{initial segment invariant}) if for all $i \leq j < \kappa$ we have
\[
\gtp(a_i, m_i; M) = \gtp(a_j, m_i; M).
\]
For $c: C \to M$ we say this is an \emph{isi-sequence over $c$} if $c$ embeds in $(M_i)_{i < \kappa}$.
\end{enumerate}
\end{definition}
\begin{convention}
\thlabel{conv:chain-of-initial-segments}
For a chain of initial segments $(M_i)_{i < \kappa}$ for some sequence $(a_i)_{i < \kappa}$ in $M$ we will abuse notation and view $a_i$ as an arrow into $M_j$ for $i < j$. Similarly, if $c$ embeds in $(M_i)_{i < \kappa}$, we view $c$ as an arrow into $M_i$ for all $i < \kappa$.
\end{convention}
\begin{lemma}
\thlabel{lem:independence-witnesses-presentability}
Suppose that $(\C, \M)$ is a $\mu$-AECat and let $\kappa \geq \mu$. Suppose furthermore that we are given a sequence $(a_i)_{i < \kappa}$ in some $M$ with a chain of initial segments $(M_i)_{i < \kappa}$ and some $c$ that embeds in this chain, such that $\dom(a_i)$ (which is the same for all $i$) and $\dom(c)$ are $\kappa$-presentable. Then there is a chain of initial segments $(M_i')_{i < \kappa}$ in which $c$ embeds, such that for all $i < \kappa$ the inclusion of $M_i'$ into $M$ factors through $M_i$ (so $M_i' \leq M_i$) and $M_i'$ is $\kappa$-presentable.
\end{lemma}
\begin{proof}
We build the chain of initial segments $(M_i')_{i < \kappa}$ by induction. For the base case we note that we can write $M_0$ as a $\kappa$-directed colimit of $\kappa$-presentable models (using \thref{prop:kappa-aecat-and-above}). As $\dom(c)$ is $\kappa$-presentable, $c$ factors through some $M_0'$ in this diagram. The successor step is similar, using that $M_i'$ and $\dom(a_i)$ are both $\kappa$-presentable and must thus factor through some $\kappa$-presentable $M_{i+1}' \leq M_{i+1}$. In the limit step we just take the colimit $M_\ell' = \colim_{i < \ell} M_i'$ and the universal property then yields an arrow $M_\ell' \to M_\ell$.
\end{proof}

\section{Lascar strong Galois types}
\label{sec:lascar-strong-galois-types}
In this section we will give a definition of Lascar strong Galois type. In the first-order setting this will coincide with Lascar strong types, see \thref{rem:lgtp-vs-gtp}. This notion will be useful later in the property \textsc{Independence Theorem} for independence relations, see \thref{def:independence-relation-advanced-properties}.

To place our definition in context, we recall a possible definition for Lascar strong types in first-order logic. Working in a monster model, tuples $a$ and $a'$ have the same Lascar strong type over $B$ if there are $a = a_0, \ldots, a_n = a'$ and models $M_1, \ldots, M_n$, each containing $B$, such that $\tp(a_i / M_{i+1}) = \tp(a_{i+1} / M_{i+1})$ for all $0 \leq i < n$.
\begin{definition}
\thlabel{def:lascar-strong-galois-type}
Let $(\C, \M)$ be an AECat with AP and fix some $((b_j)_{j \in J}; M)$. We write $((a_i)_{i \in I} / (b_j)_{j \in J}; M) \sim_{\Lgtp} ((a_i')_{i \in I} / (b_j)_{j \in J}; M)$ if there is some extension $M \to N$ and some $m_0: M_0 \to N$, where $M_0$ is a model, such that $b_j$ factors through $m_0$ for all $j \in J$ and $\gtp((a_i)_{i \in I}, m_0; N) = \gtp((a_i')_{i \in I}, m_0; N)$.

We write
\[
\Lgtp((a_i)_{i \in I} / (b_j)_{j \in J}; M) = \Lgtp((a_i')_{i \in I} / (b_j)_{j \in J}; M)
\]
for the transitive closure of $\sim_{\Lgtp}$ and we say that $((a_i)_{i \in I}; M)$ and $((a_i')_{i \in I}; M)$ have the same \emph{Lascar strong Galois type} over $(b_j)_{j \in J}$.
\end{definition}
\begin{remark}
\thlabel{rem:lgtp-vs-gtp}
By definition having the same Lascar strong Galois type is the same as having the same Lascar strong type in an AECat based on a first-order theory. That is, they coincide in any AECat of the form $(\SubMod(T), \Mod(T))$ or $(\Mod(T), \Mod(T))$ for some first-order theory $T$. We get a similar statement for continuous logic, because the same standard proofs and definitions go through.

In positive logic the situation is more subtle, but in a broad class of reasonable positive theories Lascar strong types and Lascar strong Galois types coincide. See \thref{fact:finitely-lambda-saturated-gives-lascar-type} and the surrounding discussion for more details.
\end{remark}
Classically multiple equivalent definitions are possible for Lascar strong types. We recall these and prove similar conditions for Lascar strong Galois types in Section \ref{sec:lgtp-vs-bounded-gtp}. Usually these proofs require compactness, but interestingly this can be replaced by the use of a nice enough independence relation.

For ease of notation the following proposition is formulated for single arrows, but everything goes through word for word if we replace those by tuples of arrows.
\begin{proposition}
\thlabel{prop:lascar-strong-type-invariant-under-galois-type}
Suppose that $\gtp(a_1, a_2, b; M) = \gtp(a_1', a_2', b'; M')$. Then we have $\Lgtp(a_1 / b; M) = \Lgtp(a_2 / b; M)$ iff $\Lgtp(a_1' / b'; M') = \Lgtp(a_2' / b'; M')$.
\end{proposition}
\begin{proof}
It suffices to prove that $(a_1 / b; M) \sim_{\Lgtp} (a_2 / b; M)$ implies $(a_1' / b'; M') \sim_{\Lgtp} (a_2' / b'; M')$. Let $M \to N$ with $m_0: M_0 \to N$ witness $(a_1 / b; M) \sim_{\Lgtp} (a_2 / b; M)$. So we have that $b = m_0 b^*$ for some $b^*: B \to M_0$, and $\gtp(a_1, m_0; N) = \gtp(a_2, m_0; N)$. Let $N \to N' \leftarrow M'$ witness $\gtp(a_1, a_2, b; N) = \gtp(a_1, a_2, b; M) = \gtp(a_1', a_2', b'; M')$. Then we get the following commuting diagram:
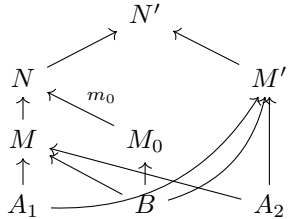
\[
\begin{tikzcd}[row sep=small]
                                       & N'                                             &                            \\
N \arrow[ru]                           &                                                & M' \arrow[lu]              \\
M \arrow[u]                            & M_0 \arrow[lu, "m_0"']                         &                            \\
A_1 \arrow[u] \arrow[rruu, bend right] & B \arrow[lu] \arrow[u] \arrow[ruu, bend right] & A_2 \arrow[llu] \arrow[uu]
\end{tikzcd}
\]
So we have $m_0: M_0 \to N'$ and $b'$ factors though $m_0$. This follows from the fact that the above diagram commutes, so $b': B \to M' \to N'$ and $B \to M_0 \to N \to N'$ are the same arrow. Furthermore, we have that $\gtp(a_1, m_0; N') = \gtp(a_2, m_0; N')$ and because $a_1$ and $a_2$ are the same arrows into $N'$ as $a_1'$ and $a_2'$ respectively, we get $\gtp(a_1', m_0; N') = \gtp(a_2', m_0; N')$. So we conclude $(a_1' / b'; M') \sim_{\Lgtp} (a_2' / b'; M')$, as required.
\end{proof}
Lascar strong Galois types induce a bounded equivalence relation on arrows. We again give a proof for single arrows, which also works for tuples of arrows.
\begin{proposition}
\thlabel{prop:lascar-strong-galois-types-bounded}
Given objects $A$ and $B$ there is $\lambda$ such that for any $b: B \to M$ the relation of having the same Lascar strong Galois type over $b$ partitions $\Hom(A, M)$ into at most $\lambda$ many equivalence classes.
\end{proposition}
\begin{proof}
We will first prove the following claim: for any $b: B \to M$ there is $\lambda_b$ such that for any $b': B \to M'$ with $\gtp(b'; M') = \gtp(b; M)$ there are at most $\lambda_b$ many equivalence classes of Lascar strong Galois types over $b'$ in $\Hom(A, M')$. By \thref{fact:galois-type-set} the collection $\Sgtp(A, M)$ is a set. We pick $\lambda_b = |\Sgtp(A, M)|$. Now let $M \to N \leftarrow M'$ witness $\gtp(b'; M') = \gtp(b; M)$. For any two arrows $a, a': A \to M'$ we have that $\gtp(a, m; N) = \gtp(a', m; N)$ implies that $\Lgtp(a / b'; M') = \Lgtp(a' / b'; M')$, by definition of $\Lgtp$. The claim then follows by choice of $\lambda_b$.

By the claim we can take $\lambda$ to be the supremum of $\lambda_b$, where $b$ ranges over the representatives of the Galois types in $\Sgtp(B)$.
\end{proof}
\begin{proposition}
\thlabel{prop:lascar-strong-galois-types-basic-observations}
If $\Lgtp((a_i)_{i \in I} / b; M) = \Lgtp((a_i')_{i \in I} / b; M)$ then:
\begin{enumerate}[label=(\roman*)]
\item (restriction) we have $\Lgtp((a_i)_{i \in I_0} / b; M) = \Lgtp((a_i')_{i \in I_0} / b; M)$ for any $I_0 \subseteq I$;
\item (monotonicity) given an arrow $c_i: C_i \to \dom(a_i)$ for each $i \in I$, then
\[
\Lgtp((a_i)_{i \in I}, (a_i c_i)_{i \in I} / b; M) = \Lgtp((a_i')_{i \in I}, (a_i' c_i)_{i \in I} / b; M)
\]
and thus $\Lgtp((a_i c_i)_{i \in I} / b; M) = \Lgtp((a_i' c_i)_{i \in I} / b; M)$;
\item (extension) for any $(c; M)$ there is an extension $M \to N$ and some $(c'; N)$ such that $\Lgtp(c, (a_i)_{i \in I} / b; N) = \Lgtp(c', (a_i')_{i \in I} / b; N)$.
\end{enumerate}
\end{proposition}
\begin{proof}
This is essentially the same \thref{fact:galois-types}, but then for Lascar strong Galois types. To prove it, apply the definition of Lascar strong Galois types to reduce to some equality of Galois types and then apply \thref{fact:galois-types}.
\end{proof}
\section{Independence relations}
\label{sec:independence-relations}
Similar to \cite[Section 6]{kamsma_kim-pillay_2020} we define an independence relation in an AECat as a ternary relation on subobjects of models. However, there will be some slight differences in our terminology, see after \thref{def:independence-relation-types}.

We write $\Sub(X)$ for the poset of subobjects of object $X$. If $A \leq B$ for $A, B \in \Sub(X)$ we may also consider $A$ to be a subobject of $B$, that is $A \in \Sub(B)$. On the other hand, we always have $X \in \Sub(X)$ as the maximal element. So we will use the notation $A \leq X$ to mean that $A$ is a subobject of $X$.
\begin{convention}
\thlabel{conv:subobjects}
We extend \thref{conv:extending-monomorphisms} to subobjects: given an extension $M \to N$ and a subobject $A \leq M$, we will view $A$ as a subobject of $N$.
\end{convention}
\begin{definition}
\thlabel{def:independence-relation}
In an AECat with AP, an \emph{independence relation} is a relation on triples of subobjects of models. If such a triple $(A, B, C)$ of subobjects of a model $M$ is in the relation, we call it \emph{independent} and write:
\[
A \ind_C^M B.
\]
This notation should be read as ``$A$ is independent from $B$ over $C$ (in $M$)''.

We also allow each of the subobjects in the notation to be replaced by an arrow representing them. For example, if $a$ is an arrow representing the subobject $A$ then $a \ind_C^M B$ means $A \ind_C^M B$.
\end{definition}
We may want to restrict the objects that can appear in the base of the independence relation.
\begin{definition}
\thlabel{def:independence-base-class}
Let $(\C, \M)$ be an AECat with AP and let $\B$ be a collection of objects in $\C$, closed under isomorphic objects, with $\M \subseteq \B$. Then we call $\B$ a \emph{base class}. An independence relation $\ind$ is called an \emph{independence relation over $\B$} if it only allows subobjects with their domain in $\B$ in the base. That is, $A \ind_C^M B$ implies that the domain of $C$ is in $\B$. We will also say that $\B$ is \emph{the base class of $\ind$}, written as $\B = \base(\ind)$.
\end{definition}
\begin{convention}
\thlabel{conv:subobject-in-base-class}
For a base class $\B$ and some subobject $C \leq M$ we will also write $C \in \B$ to mean that the domain of $C$ is in $\B$, and similarly for $C \not \in \B$.
\end{convention}
\begin{definition}
\thlabel{def:basic-independence-relation}
We call an independence relation $\ind$ a \emph{basic independence relation} if it satisfies the following properties.
\begin{description}
\item[\textsc{Invariance}] $a \ind_c^M b$ and $\gtp(a, b, c; M) = \gtp(a', b', c'; M')$ implies $a' \ind_{c'}^{M'} b'$.

\item[\textsc{Monotonicity}] $A \ind_C^M B$ and $A' \leq A$ implies $A' \ind_C^M B$.

\item[\textsc{Transitivity}] $A \ind_B^M C$ and $A \ind_C^M D$ with $B \leq C$ implies $A \ind_B^M D$.

\item[\textsc{Symmetry}] $A \ind_C^M B$ implies $B \ind_C^M A$.

\item[\textsc{Existence}] $A \ind_C^M C$ for all $C \in \base(\ind)$.

\item[\textsc{Extension}] If $a \ind_c^M b$ and $(b'; M)$ is such that $b$ factors through $b'$ then there is an extension $M \to N$ with some $(a'; N)$ such that $\gtp(a', b, c; N) = \gtp(a, b, c; M)$ and $a' \ind_c^N b'$.

\item[\textsc{Union}] Let $(B_i)_{i \in I}$ be a directed system with a cocone into some model $M$, and suppose $B = \colim_{i \in I} B_i$ exists. Then if $A \ind_C^M B_i$ for all $i \in I$, we have $A \ind_C^M B$.
\end{description}
\end{definition}
Before we define some additional properties for independence relations, we first need to translate the notion of a club set to categorical language.
\begin{definition}
\thlabel{def:aecat-poset-of-submodels}
Let $(\C, \M)$ be an AECat. For a model $M$ and a regular cardinal $\kappa$ we write $\Sub^\kappa_\M(M)$ for the poset of $\kappa$-presentable subobjects of $M$ in $\M$.
\end{definition}
Note that if we are given a chain $(M_i)_{i < \theta}$ in $\Sub^\kappa_\M(M)$ with $\theta < \kappa$ then its join in $\Sub^\kappa_\M(M)$ exists and is given by $\colim_{i < \theta} M_i$. This is the reason why we restrict ourselves to $\M$, because there we have directed colimits. If $\C$ has directed colimits as well then all these definitions would make sense for $\C$ as well.
\begin{definition}
\thlabel{def:categorical-club-set}
Let $\F \subseteq \Sub^\kappa_\M(M)$ be a nonempty set.
\begin{enumerate}[label=(\roman*)]
\item We call $\F$ \emph{unbounded} if for every $M_0 \in \Sub^\kappa_\M(M)$ there is $M_1 \in \F$ such that $M_0 \leq M_1$.
\item We call $\F$ \emph{closed} if for any chain $(M_i)_{i < \theta}$ in $\F$ with $\theta < \kappa$ its join $\colim_{i < \theta} M_i$ is again in $\F$.
\item We call $\F$ a \emph{club set} if it is closed and unbounded.
\end{enumerate}
\end{definition}
\begin{fact}
\thlabel{fact:categorical-club-sets}
The following two facts are standard.
\begin{enumerate}[label=(\roman*)]
\item The intersection of two club sets on $\Sub^\kappa_\M(M)$ is again a club set.
\item If $M = \colim_{i < \kappa} M_i$, where $(M_i)_{i < \kappa}$ is a continuous chain of $\kappa$-presentable models, then $\{M_i : i < \kappa\}$ is a club set on $\Sub^\kappa_\M(M)$.
\end{enumerate}
\end{fact}
\begin{proof}
Fact (i) is standard, see for example \cite[Theorem 8.2]{jech_set_2003}. We just apply the argument to the poset $\Sub^\kappa_\M(M)$ instead of to a cardinal considered as a poset. Fact (ii) is just unfolding definitions. The chain $(M_i)_{i < \kappa}$ is unbounded  because it is $\kappa$-directed, so any $\kappa$-presentable $M' \leq M$ will factor through the chain, and continuity is precisely saying that the chain is a closed set.
\end{proof}
\begin{definition}
\thlabel{def:independence-relation-advanced-properties}
We also define the following properties for an independence relation.
\begin{description}
\item[\textsc{Base-Monotonicity}] $A \ind_C^M B$ and $C \leq C' \leq B$ with $C' \in \base(\ind)$ implies $A \ind_{C'}^M B$.

\item[\textsc{Club Local Character}] For every regular cardinal $\lambda$ there is a regular cardinal $\Upsilon(\lambda)$ such that the following holds for all regular $\kappa \geq \Upsilon(\lambda)$. Let $A, M \leq N$, with $A$ $\lambda$-presentable and $M$ a model. Then there is a club set $\F \subseteq \Sub^\kappa_\M(M)$ such that for all $M_0 \in \F$ we have $A \ind_{M_0}^N M$.

\item[\textsc{Stationarity}] If $\gtp(a, m; N) = \gtp(a', m; N)$, where the domain of $m$ is a model, then $a \ind_m^N b$ and $a' \ind_m^N b$ implies $\gtp(a, m, b; N) = \gtp(a', m, b; N)$.

\item[\textsc{Independence Theorem}] Suppose we have $a \ind_c^M b$, $a' \ind_c^M b'$, $b \ind_c^M b'$ and also $\Lgtp(a / c; M) = \Lgtp(a' / c; M)$. Then there is an extension $M \to N$ with $(a^*; N)$ such that $\Lgtp(a^*, b / c; N) = \Lgtp(a, b / c; N)$, $\Lgtp(a^*, b' / c; N) = \Lgtp(a', b' / c; N)$ and  $a^* \ind_c^N M$.
\end{description}
\end{definition}
\begin{definition}
\thlabel{def:independence-relation-types}
Let $\ind$ be a basic independence relation. We call $\ind$ ...
\begin{itemize}
\item ... a \emph{stable independence relation} if it also satisfies \textsc{Base-Monotonicity}, \textsc{Club Local Character}, \textsc{Stationarity} and \textsc{Independence Theorem}.
\item ... a \emph{simple independence relation} if it also satisfies \textsc{Base-Monotonicity}, \textsc{Club Local Character} and \textsc{Independence Theorem}.
\item ... an \emph{NSOP$_1$-like independence relation} if it also satisfies \textsc{Club Local Character} and \textsc{Independence Theorem};
\end{itemize}
\end{definition}
We briefly compare the terminology we use here to the terminology in \cite{kamsma_kim-pillay_2020}. The notion of a base class is entirely new, so properties concerning the subobject in the base have been adjusted accordingly.

The \textsc{Existence} and \textsc{Transitivity} properties are different now, and we have added \textsc{Extension}. However, the formulation of existence and transitivity (together with invariance and monotonicity) in \cite{kamsma_kim-pillay_2020} implies our new formulation, see \cite[Proposition 6.11]{kamsma_kim-pillay_2020}. The converse is also true: our new formulation of \textsc{Existence}, \textsc{Transitivity} and \textsc{Extension} (together with invariance) implies existence and transitivity there. So ultimately the two approaches are equivalent.
\begin{remark}
\thlabel{rem:local-character}
In a more traditional definition of local character, such as in simple theories, one would just require that for $A, M \leq N$ as above there is some $\Upsilon(\lambda)$-presentable $M_0 \leq M$ such that $A \ind_{M_0}^N M$. This is (almost) the definition that was used in \cite{kamsma_kim-pillay_2020}, where we also have access to \textsc{Base-Monotonicity}. We then get \textsc{Club Local Character} by considering the club set $\F = \{ M_0 \in \Sub^\kappa_\M(M) : M_0 \leq M \}$. In NSOP$_1$-like settings we do generally not have \textsc{Base-Monotonicity}. So \textsc{Club Local Character} then still gives us a good amount of \textsc{Base-Monotonicity}, namely on a club set. These ideas are due to \cite{kaplan_local_2019}.
\end{remark}
\begin{proposition}[Strong extension]
\thlabel{prop:strong-extension}
Let $\ind$ be a basic independence relation and suppose that $a \ind_c^M b$. Then for any $(d; M)$ there is an extension $M \to N$ and $(d'; N)$ such that $\Lgtp(d'/b,c; N) = \Lgtp(d/b,c; N)$ and $a \ind_c^N d'$.
\end{proposition}
\begin{proof}
We first apply \textsc{Extension} to find $M \to N_1$ with $m': M \to N_1$ such that $a \ind_c^{N_1} m'$ and $\gtp(m', b, c; N_1) = \gtp(m, b, c; N_1)$. In particular this means that $b$ and $c$ factor through $m'$ by \thref{fact:galois-type-factorisation}. We apply \textsc{Extension} again to find an extension $n_1: N_1 \to N$ and $n_1': N_1 \to N$ with $a \ind_c^N n_1'$ and $\gtp(n_1', m'; N) = \gtp(n_1, m'; N)$. We define $d'$ to be the composition $D \xrightarrow{d} M \to N_1 \xrightarrow{n_1'} N$. By \textsc{Monotonicity} we then have $a \ind_c^N d'$. We also have $\gtp(d', m'; N) = \gtp(d, m'; N)$, so since $b$ and $c$ factor through $m'$, and $m'$ has a model as domain, we indeed get $\Lgtp(d'/b,c; N) = \Lgtp(d/b,c; N)$.
\end{proof}
\begin{corollary}
\thlabel{cor:strong-extension-dual}
Let $\ind$ be a basic independence relation and suppose that $a \ind_c^M b$. Then for any $(d; M)$ there is $M \to N$ and $(a'; N)$ such that $\Lgtp(a'/b,c; N) = \Lgtp(a/b,c; N)$ and $a' \ind_c^N d$.
\end{corollary}
\begin{proof}
Apply \thref{prop:strong-extension} to find $M \to N'$ with $(d'; N')$ such that $a \ind_c^{N'} d'$ and $\Lgtp(d'/b,c; N') = \Lgtp(d/b,c; N')$. Then just pick $(a'; N)$ in an extension $N' \to N$ such that $\Lgtp(a',d/b,c; N) = \Lgtp(a,d'/b,c; N)$.
\end{proof}
\begin{convention}
\thlabel{conv:local-character-function}
We call the class function $\Upsilon$ for \textsc{Club Local Character} a \emph{local character function}. For an object $A$ we write $\Upsilon(A)$ for $\Upsilon(\lambda)$ where $\lambda$ is the least regular cardinal such that $A$ is $\lambda$-presentable.
\end{convention}
\begin{lemma}[Chain local character]
\thlabel{lem:chain-local-character}
Let $\ind$ be an independence relation satisfying \textsc{Club Local Character}. Let $A \leq N$ and $\kappa \geq \Upsilon(A)$. Suppose that we are given a continuous chain $(M_i)_{i < \kappa}$ of $\kappa$-presentable models with $M = \colim_{i < \kappa} M_i \leq N$. Then there is $i_0 < \kappa$ such that $A \ind_{M_{i_0}}^N M$.
\end{lemma}
\begin{proof}
Let $\F \subseteq \Sub^\kappa_\M(M)$ be the club set from \textsc{Club Local Character}. By \thref{fact:categorical-club-sets}(ii) the chain $(M_i)_{i < \kappa}$ forms a club set on $\Sub^\kappa_\M(M)$. So by \thref{fact:categorical-club-sets}(i) $\{ M_i : i < \kappa \} \cap \F$ is nonempty.
\end{proof}
\begin{remark}
\thlabel{rem:only-chain-local-character}
For all our results we only need chain local character. That is, the conclusion of \thref{lem:chain-local-character}. In particular the canonicity theorems in Section \ref{sec:canonicity} go through even if we would just assume chain local character.
\end{remark}
Given that we actually only need chain local character, as per \thref{rem:only-chain-local-character}, it is natural to ask whether the converse of \thref{lem:chain-local-character} holds. That is, if chain local character implies \textsc{Club Local Character}. This is not so clear, so we leave it at this.

We recall the following from \cite[Definition 6.13]{kamsma_kim-pillay_2020}.\footnote{There is a slight improvement here. We additionally require here that $c$ embeds in the chain. In \cite{kamsma_kim-pillay_2020} this is the case for every such sequence that is considered. There are only two uses of such sequences. One in Lemma 6.14 where they are constructed, and we indeed get that $c$ embeds in the chain of witnesses of independence. The other one is in the main proof, where such a sequence is found from an application of Lemma 6.14.}
\begin{definition}
\thlabel{def:independent-sequence}
Suppose we have an independence relation $\ind$. Let $(a_i)_{i < \kappa}$ be a sequence in some $M$ and let $c: C \to M$ be an arrow. Suppose that $(M_i)_{i < \kappa}$ is a chain of initial segments for $(a_i)_{i < \kappa}$ and that $c$ embeds in the chain. Then we call $(M_i)_{i < \kappa}$ \emph{witnesses of $\ind_c$-independence}\index{witnesses of independence} for $(a_i)_{i < \kappa}$ if
\[
a_i \ind_c^M M_i
\]
for all $i < \kappa$. We say that a sequence is \emph{$\ind_c$-independent}\index{independent sequence} if it admits a chain of witnesses of $\ind_c$-independence.
\end{definition}
The following proposition is the standard argument showing that we can find arbitrarily long independent sequences, assuming very few properties for our independence relation (see e.g.\ \cite[Proposition 2.2.4]{kim_simplicity_2014}). The proposition after that shows that if we additionally assume \textsc{Union} we can actually get arbitrarily long independent isi-sequences.
\begin{proposition}
\thlabel{prop:build-independent-sequence}
Let $\ind$ be an independence relation satisfying \textsc{Invariance}, \textsc{Existence} and \textsc{Extension}. Then for any $(a, c; M)$ with $\dom(c) \in \base(\ind)$ and any $\kappa$ there is some extension $M \to N$ containing a $\ind_c$-independent sequence $(a_i)_{i < \kappa}$ with $\gtp(a_i, c; N) = \gtp(a, c; M)$ for all $i < \kappa$.
\end{proposition}
\begin{proof}
We construct the witnesses of independence $(M_i)_{i < \kappa}$ and sequence $(a_i)_{i < \kappa}$ by induction. At stage $i$ we will construct $a_i$ and $M_{i+1}$. By \textsc{Existence} we have $a \ind_c^M c$, and so we will have $a \ind_c^{M_i} c$ for all $i < \kappa$. At every stage we will apply \textsc{Extension} to the latter.

For the base case we set $M_0 = M$ and use \textsc{Extension} to find $a_0$ and $M \to M_1$ with $\gtp(a_0, c; M_1) = \gtp(a, c; M)$ and $a_0 \ind_c^{M_1} M_0$. In the successor step we use \textsc{Extension} to find $M_{i+1} \to M_{i+2}$ and $a_{i+1}$ such that $a_{i+1} \ind_c^{M_{i+2}} M_{i+1}$ and $\gtp(a_{i+1}, c; M_{i+2}) = \gtp(a, c; M)$. Finally, for limit $\ell < \kappa$ let $M_\ell = \colim_{i < \ell} M_i$. We use \textsc{Extension} to find $M_\ell \to M_{\ell+1}$ and $a_\ell$ with $\gtp(a_\ell, c; M_{\ell+1}) = \gtp(a, c; M)$ and $a_\ell \ind_c^{M_{\ell+1}} M_\ell$. We finish the construction by taking $N = \colim_{i < \kappa} M_i$.
\end{proof}
\begin{proposition}
\thlabel{prop:build-independent-isi-sequence}
Suppose that $\ind$ is an independence relation satisfying \textsc{Invariance}, \textsc{Existence}, \textsc{Extension} and \textsc{Union}. Then given $(a, c; M)$ with $\dom(c) \in \base(\ind)$ and any $\kappa$, there is a $\ind_c$-independent isi-sequence $(a_i)_{i < \kappa}$ over $c$ in some extension $M \to N$ such that $\gtp(a_i, c; N) = \gtp(a, c; M)$ for all $i < \kappa$.
\end{proposition}
\begin{proof}
This is just \cite[Lemma 6.14]{kamsma_kim-pillay_2020}. The differences in terminology are discussed after \thref{def:independence-relation-types}. In particular, the monotonicity assumption there is only necessary to get what we call \textsc{Extension}. Finally, there is no notion of base class there, but this is only relevant in the application of \textsc{Existence}, which is only applied with $c$ in the base. Hence the assumption $\dom(c) \in \base(\ind)$.
\end{proof}
\section{Independence theorem, 3-amalgamation and stationarity}
\label{sec:independence-theorem-3-amalgamation-and-stationarity}
It is well known that the property \textsc{Independence Theorem} can also be formulated as an amalgamation property of some independent system. This allows for a more categorical statement without any mention of Lascar strong Galois types. However, we need to restrict ourselves to work only over models. We will give this property its own name and prove its equivalence to \textsc{Independence Theorem}, modulo some basic properties, in \thref{thm:independence-theorem-vs-3-amalgamation}.

The contents of this section are not necessary for the results in the rest of this paper, but we do refer to them a few times in remarks and discussions.
\begin{definition}[{\cite[Definition 6.7]{kamsma_kim-pillay_2020}}]
\thlabel{def:3-amalgamation}
An independence relation $\ind$ has \textsc{3-amalgamation} if the following holds. Suppose that we have
\[
A \ind_M^{N_1} B, \quad
B \ind_M^{N_2} C, \quad
C \ind_M^{N_3} A,
\]
overloading notation for subobjects of different models. Suppose furthermore that $M$ is a model and that
\begin{align*}
\gtp(a, m; N_1) &= \gtp(a, m; N_3), \\
\gtp(b, m; N_1) &= \gtp(b, m; N_2), \\
\gtp(c, m; N_2) &= \gtp(c, m; N_3),
\end{align*}
where $a$, $b$, $c$ and $m$ are representatives for the subobjects $A$, $B$, $C$ and $M$ respectively (again, overloading notation for different models). Then we can find extensions from $N_1$, $N_2$ and $N_3$ to some $N$ such that the diagram we obtain in that way commutes:
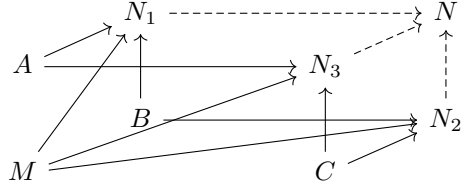
\[
\begin{tikzcd}[row sep=tiny]
 & N_1 \arrow[rrr, dashed] &  &  & N \\
A \arrow[rrr] \arrow[ru] &  &  & N_3 \arrow[ru, dashed] &  \\
 & B \arrow[uu] \arrow[rrr] &  &  & N_2 \arrow[uu, dashed] \\
M \arrow[rrruu] \arrow[rrrru] \arrow[ruuu] &  &  & C \arrow[ru] \arrow[uu] & 
\end{tikzcd}
\]
Furthermore, these extensions are such that $A \ind_M^N N_2$.
\end{definition}
\begin{theorem}
\thlabel{thm:independence-theorem-vs-3-amalgamation}
Let $\ind$ be a basic independence relation. If $\ind$ satisfies \textsc{Independence Theorem} then it also satisfies \textsc{3-amalgamation}. Conversely, if $\ind$ satisfies \textsc{3-amalgamation} then it satisfies \textsc{Independence Theorem} over models (i.e.\ we require the base $C$ to be a model).
\end{theorem}
\begin{proof}
We first prove that \textsc{Independence Theorem} implies \textsc{3-amalgamation}. Let the set up be as in \thref{def:3-amalgamation}. In the diagram below we find the dashed arrows by using $\gtp(c, m; N_2) = \gtp(c, m; N_3)$ and $\gtp(b, m; N_1) = \gtp(b, m; N_2)$.
\[
\begin{tikzcd}[row sep=tiny]
                                           & N_1 \arrow[rrr, dashed]  &  &                         & N'                        \\
A \arrow[rrr] \arrow[ru]                   &                          &  & N_3 \arrow[r, dashed]   & \bullet \arrow[u, dashed] \\
                                           & B \arrow[rrr] \arrow[uu] &  &                         & N_2 \arrow[u, dashed]     \\
M \arrow[rrrru] \arrow[ruuu] \arrow[rrruu] &                          &  & C \arrow[ru] \arrow[uu] &                          
\end{tikzcd}
\]
We write $a_1$ for the arrow $A \to N_1 \to N'$ and $a_3$ for the arrow $A \to N_3 \to N'$. Then we have $\Lgtp(a_1 / m; N') = \Lgtp(a_3 / m; N')$. We can thus apply \textsc{Independence Theorem} to find some extension $N' \to N^*$ with some $a^*: A \to N^*$ such that $\Lgtp(a^*,b/m; N^*) = \Lgtp(a_1,b/m; N^*)$, $\Lgtp(a^*,c/m; N^*) = \Lgtp(a_3,c/m; N^*)$ and $a^* \ind_M^{N^*} N'$. So in particular we have $a^* \ind_M^{N^*} N_2$ by \textsc{Monotonicity}. Using $\gtp(a, b, m; N_1) = \gtp(a_1, b, m; N^*) = \gtp(a^*, b, m; N^*)$ and $\gtp(a, c, m; N_3) = \gtp(a_3, c, m; N^*) = \gtp(a^*, c, m; N^*)$ after each other we find an extension $N^* \to N$ together with extensions from $N_1$ and $N_3$ to $N$ and we just forget about the previous extensions from $N_1$ and $N_3$ to $N^*$. These two new extensions, together with $N_2 \to N^* \to N$, then form the solution to our \textsc{3-amalgamation} problem.

Now we prove the converse. So we assume \textsc{3-amalgamation} and we prove \textsc{Independence Theorem} over models. So suppose that $a \ind_m^N b$, $a' \ind_m^N c$ and $b \ind_m^N c$ and $\Lgtp(a / m; N) = \Lgtp(a' / m; N)$. Then we can form the commuting diagram as below, where we find the dashed arrows by \textsc{3-amalgamation}.
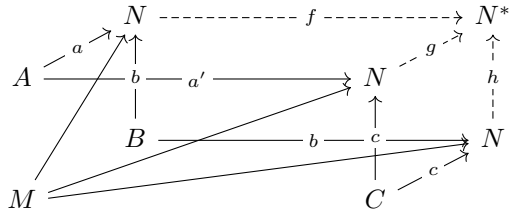
\[
\begin{tikzcd}[row sep=small]
                                                             & N \arrow[rrrr, "f" description, dashed]                     &  &  &                                                           & N^*                                   \\
A \arrow[rrrr, "a'" description] \arrow[ru, "a" description] &                                                             &  &  & N \arrow[ru, "g" description, dashed]                     &                                       \\
                                                             & B \arrow[uu, "b" description] \arrow[rrrr, "b" description] &  &  &                                                           & N \arrow[uu, "h" description, dashed] \\
M \arrow[rrrruu] \arrow[rrrrru] \arrow[ruuu]                 &                                                             &  &  & C \arrow[ru, "c" description] \arrow[uu, "c" description] &                                      
\end{tikzcd}
\]
We take the extension $N \to N^*$ to be $h$ and write $a^*$ for $fa = ga'$. The application of \textsc{3-amalgamation} yields $a^* \ind_m^{N^*} N$. Furthermore, $\gtp(a^*, b, m; N^*) = \gtp(a, b, m; N) = \gtp(a, b, m; N^*)$, so $\Lgtp(a^*, b / m; N^*) = \Lgtp(a, b / m; N^*)$ because the domain of $m$ is a model. Similarly, we also find $\Lgtp(a^*, c / m; N^*) = \Lgtp(a', c / m; N^*)$, which concludes the proof.
\end{proof}
We recall that \textsc{3-amalgamation} follows from \textsc{Stationarity}. This might for example be useful in the case where $\base(\ind) = \M$, so one does not need to verify the \textsc{Independence Theorem} property for a stable independence relation, as it follows automatically. In particular this means that a stable independence relation in the sense of \cite{lieberman_forking_2019} also yields a stable independence relation in our setting and vice versa, see also \cite[Remark 6.7]{kamsma_kim-pillay_2020} for further comparison.
\begin{fact}[{\cite[Proposition 6.16]{kamsma_kim-pillay_2020}}]
\thlabel{fact:stationarity-implies-3-amalgamation}
Let $\ind$ be a basic independence relation satisfying \textsc{Stationarity} then it also satisfies \textsc{3-amalgamation}.
\end{fact}
\section{Long dividing, isi-dividing and long Kim-dividing}
\label{sec:long-dividing}
In this section we introduce various notions of dividing, each yielding its own independence relation. These notions are based on the classical notion of dividing, as we know it from first-order logic. For the convenience of the reader, and to compare it to the new definitions, we recall the definition of dividing.
\begin{definition}
\thlabel{def:classical-dividing}
In the classical setting of first-order logic, we say that a type $p(x, b) = \tp(a/Cb)$ \emph{divides over $C$} if there is a $C$-indiscernible sequence $(b_i)_{i < \omega}$ such that $\tp(b_i/C) = \tp(b/C)$ for all $i < \omega$ and $\bigcup_{i < \omega} p(x, b_i)$ is inconsistent.
\end{definition}
In many proofs, to use this definition, one has to apply compactness in one way or another. For example to elongate the sequence $(b_i)_{i < \omega}$, or to find a finite subsequence along which $p$ is inconsistent. This is generally an issue in AECats, because we do not have compactness there. To solve this we introduce the notion of long dividing. The name is due to the fact that we consider arbitrarily long sequences in the definition, something that we would normally have to use compactness for. This is very close to the notion of isi-dividing from \cite[Definition 5.7]{kamsma_kim-pillay_2020}. In fact, isi-dividing is just long dividing but then restricted to isi-sequences, so that we have some homogeneity in the sequences involved. Of course, indiscernible sequences would be even more homogeneous, but the little bit that isi-sequences offer us turns out to be enough.

We remind the reader of \thref{conv:regular-cardinals}, namely that all cardinals considered are regular. So in the following definition we only quantify over regular cardinals.
\begin{definition}
\thlabel{def:long-dividing}
Let $(\C, \M)$ be an AECat with AP and fix some $(a, b, c; M)$.
\begin{enumerate}[label=(\roman*)]
\item Suppose that we have some sequence $((b_i)_{i \in I}; M)$ such that $\gtp(b_i, c; M) = \gtp(b, c; M)$ for all $i \in I$. We say that $\gtp(a, b, c; M)$ is \emph{consistent} for $(b_i)_{i \in I}$ if there is an extension $M \to N$ and an arrow $(a'; N)$ such that
\[
\gtp(a, b, c; M) = \gtp(a', b_i, c; N)
\]
for all $i \in I$. We call $a'$ a \emph{realisation} of $\gtp(a, b, c; M)$ for $(b_i)_{i \in I}$.

Being inconsistent is the negation of the above. So $\gtp(a, b, c; M)$ is \emph{inconsistent} for $(b_i)_{i \in I}$ if there is no extension of $M$ with a realisation $a'$ of $\gtp(a, b, c; M)$ for $(b_i)_{i \in I}$.
\item We say that $\gtp(a, b, c; M)$ \emph{long divides} over $c$ if there is $\mu$ such that for every $\lambda \geq \mu$ there is a sequence $(b_i)_{i < \lambda}$ in some extension $M \to N$ with $\gtp(b_i, c; N) = \gtp(b, c; M)$ for all $i < \lambda$, such that for some $\kappa < \lambda$ and every $I \subseteq \lambda$ with $|I| = \kappa$ we have that $\gtp(a, b, c; M)$ is inconsistent for $(b_i)_{i \in I}$.
\item We say that $\gtp(a, b, c; M)$ \emph{isi-divides} if it long divides with respect to isi-sequences over $c$. That is, we require the sequence $(b_i)_{i < \lambda}$ to be an isi-sequence over $c$.
\end{enumerate}
\end{definition}
\begin{remark}
\thlabel{rem:long-dividing-vs-dividing}
We discussed how long dividing and isi-dividing are inspired by dividing. A natural question would be whether or not they are actually the same. This was discussed in \cite{kamsma_kim-pillay_2020} from Remark 5.8 and onwards. The discussion there is just about isi-dividing but applies to long dividing as well. The summary is as follows, restricting ourselves to AECats obtained from a first-order theory.
\begin{enumerate}[label=(\roman*)]
\item Dividing implies isi-dividing implies long dividing.
\item Long dividing (and thus isi-dividing) implies dividing if we assume the existence of a proper class of Ramsey cardinals.
\item Using the canonicity theorem, \thref{thm:canonicity-of-simple-independence}, we can actually drop the large cardinal assumption and conclude that isi-dividing implies dividing in simple theories.
\item The question remains: does long dividing (and thus isi-dividing) generally imply dividing, without the large cardinal assumption?
\end{enumerate}
In fact, dividing makes sense in positive logic as well (see \cite{pillay_forking_2000, ben-yaacov_simplicity_2003}), and the above discussion applies to AECats obtained from a positive theory as well. There is an even more general setting, namely that of finitely short AECats, where this discussion applies. These are AECats where a Galois type of an infinite tuple of arrows is determined by the Galois types of the finite subtuples, see \cite[Section 4]{kamsma_independence_2021}. This is more or less the same framework as that of homogeneous model theory in the sense of \cite{buechler_simple_2003} (see \cite[Example 4.4]{kamsma_kim-pillay_2020}).
\end{remark}
The point of these dividing notions is that they yield independence relations. We can already prove some basic properties about these independence relations in arbitrary AECats, similar to the basic properties that dividing always has.
\begin{proposition}
\thlabel{prop:long-dividing-subobjects}
Let $A, B, C \leq M$ be subobjects. Let $(a, b, c; M)$ and $(a', b', c'; M)$ be two sets of representatives. Then $\gtp(a, b, c; M)$ long divides over $c$ if and only if $\gtp(a', b', c'; M)$ long divides over $c'$. The same statement holds for isi-dividing.
\end{proposition}
\begin{proof}
This comes down to checking all the definitions, which is lengthy to do in detail. However, there is only one trick that we repeatedly use, and that is \thref{fact:galois-types}(ii). To apply this trick we let $f, g, h$ be isomorphisms such that $a' = af$, $b' = bg$ and $c' = ch$. Then, using the above trick, we easily see that for any sequence $(b_i)_{i < \lambda}$ witnessing long dividing of $\gtp(a, b, c; M)$ we have that $(b_i g)_{i < \lambda}$ witnesses long dividing for $\gtp(a, b g, c; M) = \gtp(a, b', c; M)$. Similarly we can replace $c$ by $ch = c'$ and $a$ by $af = a'$. The same holds for isi-dividing, noting that any isi-sequence over $c$ is also an isi-sequence over $ch = c'$.
\end{proof}
\begin{definition}
\thlabel{def:long-dividing-independence}
For subobjects $A, B, C \leq M$ we write $A \ind_C^{\ld,M} B$ if $\gtp(a, b, c; M)$ does not long divide for all (equivalently: some) representatives $a, b, c$ of $A, B, C$. Similarly, we write $A \ind_C^{\isid,M} B$ if $\gtp(a, b, c; M)$ does not isi-divide.
\end{definition}
As $\ind^\ld$ and $\ind^\isid$ do not generally satisfy \textsc{Symmetry} we have to distinguish between ``left'' and ``right'' versions of certain properties, as we do below.
\begin{proposition}
\thlabel{prop:basic-properties-long-dividing-and-isi-dividing}
Long dividing and isi-dividing always satisfy the following: \textsc{Invariance}, \textsc{Left-Monotonicity}, \textsc{Existence} and \textsc{Base-Monotonicity}. In addition, long dividing also satisfies \textsc{Right-Monotonicity}
\end{proposition}
\begin{proof}
Everything is direct from the definition, except for \textsc{Right-Monotonicity} for long dividing and \textsc{Base-Monotonicity}. For both we will prove the contraposition.

For \textsc{Base-Monotonicity} let $(a, b, c, c'; M)$ be such that $\gtp(a, b, c'; M)$ long divides over $c'$ and $C \leq C' \leq B$, where $C, C', B$ are the subobjects represented by $c, c', b$ respectively. Let $\mu$ be as in the definition of long dividing and let $\lambda \geq \mu$. Then there is $(b_i)_{i < \lambda}$ in some $N$ that witnesses long dividing of $\gtp(a, b, c'; M)$ over $c'$. We will prove that it also witnesses long dividing of $\gtp(a, b, c; M)$. Indeed we have for all $i < \lambda$ that $\gtp(b_i, c'; N) = \gtp(b, c'; M)$ and thus $\gtp(b_i, c; N) = \gtp(b, c; M)$, because $C \leq C'$. Let $\kappa < \lambda$ be such that for $I \subseteq \lambda$ with $|I| = \kappa$ we have that $\gtp(a, b, c'; M)$ is inconsistent for $(b_i)_{i \in I}$. We claim that for such $I$ we also have that $\gtp(a, b, c; M)$ is inconsistent for $(b_i)_{i \in I}$. Suppose that there would be a realisation $(a'; N')$ for some extension $N \to N'$, then $\gtp(a', b_i, c; N') = \gtp(a, b, c; M)$ for all $i \in I$. Since $c'$ factors through $b$ and $b_i$ in the same way for all $i < \lambda$, see \thref{fact:galois-type-factorisation}, we then have $\gtp(a', b_i, c'; N') = \gtp(a, b, c'; M)$ for all $i \in I$, contradicting that $\gtp(a, b, c'; M)$ is inconsistent for $(b_i)_{i \in I}$. This proves \textsc{Base-Monotonicity} for long dividing. We have shown that the same sequences that witness long dividing of $\gtp(a, b, c'; M)$ also witness long dividing of $\gtp(a, b, c; M)$. As any isi-sequence over $c'$ is an isi-sequence over $c$, the same proof shows that isi-dividing has \textsc{Base-Monotonicity}.

Now we prove \textsc{Right-Monotonicity} for long dividing. Let $(a, b, b', c; M)$ be such that $\gtp(a, b, c; M)$ long divides over $c$ and $b$ factors through $b'$. For any sequence $(b_i)_{i < \lambda}$ in some $N$ witnessing long dividing we can form $(b_i')_{i < \lambda}$ by letting $b_i'$ be such that $\gtp(b_i', b_i, c; N) = \gtp(b', b, c; M)$ for all $i < \lambda$ (possibly replacing $N$ by an extension in the process). Then for $I \subseteq \lambda$ a realisation of $\gtp(a, b', c; M)$ for $(b_i')_{i \in I}$ would also be a realisation of $\gtp(a, b, c; M)$ for $(b_i)_{i \in I}$. So if we let $\kappa < \lambda$ be such that for every $I \subseteq \lambda$ with $|I| = \kappa$ we have that $\gtp(a, b, c; M)$ is inconsistent for $(b_i)_{i \in I}$, we also get that $\gtp(a, b', c; M)$ is inconsistent for $(b_i')_{i \in I}$ for any such $I$. We conclude that $\gtp(a, b', c; M)$ long divides over $c$.
\end{proof}
We note that in the above proof we did not have to change the sequence involved for \textsc{Base-Monotonicity}, which was why the same proof also works for isi-dividing. In the proof of \textsc{Right-Monotonicity} we had to build a new sequence, which might not be an isi-sequence again. This is why that proof only works for long dividing.

An important property that misses in \thref{prop:basic-properties-long-dividing-and-isi-dividing} is \textsc{Extension}. Classically (working in a monster model) this is fixed by considering forking instead of dividing. This forces the \textsc{Extension} property as follows. Suppose that \textsc{Extension} fails for some type $p = \tp(a/Cb)$. Then there is some parameter set, say $D$, such that every extension of $p$ to $D$ divides over $C$. In other words, $p$ implies a disjunction of types over $D$ such that every type in that disjunction divides over $C$. Classically we could even further reformulate this by using compactness and having $p$ actually imply a finite disjunction of dividing formulas, but that is not necessary and we want to avoid compactness in our definitions. So our definition of isi-forking will be the semantical way of saying ``implies a (possibly infinite) disjunction of types that each isi-divide''.
\begin{definition}
\thlabel{def:isi-forking}
We say that $\gtp(a, b, c; M)$ \emph{isi-forks} over $c$ if there is some extension $M \to N$ with $((a_j)_{j \in J}, (d_j)_{j \in J}; N)$ such that:
\begin{enumerate}[label=(\roman*)]
\item $\gtp(a_j, d_j, c; N)$ isi-divides over $c$ for each $j \in J$;
\item given an extension $N \to N'$ with some $(a'; N')$ such that $\gtp(a', b, c; N') = \gtp(a, b, c ; N)$ there is $j \in J$ such that $\gtp(a', d_j, c; N') = \gtp(a_j, d_j, c; N)$.
\end{enumerate}
\end{definition}
Note that we do not require that $b$ actually factors through the $d_j$. This is because we also want to force in \textsc{Right-Monotonicity}.

Of course, one could also define a notion of \emph{long forking} by replacing isi-dividing by long dividing in the above. However, we will have no use for this.
\begin{remark}
\thlabel{rem:isi-forking-vs-forking}
The definition of isi-forking is just the semantical way of saying ``$\gtp(a, b, c; M)$ implies a (possibly infinite) disjunction of Galois types that each isi-divide over $c$''. In the first-order setting and in the positive setting (see \cite{pillay_forking_2000}) forking has been defined and can be formulated as follows: a type forks over $C$ if it implies a (possibly infinite) disjunction of types that each divide over $C$. It should then be clear that forking implies isi-forking. This uses the fact that dividing implies isi-dividing, see \thref{rem:long-dividing-vs-dividing}(i). If isi-dividing and dividing coincide then the converse is true, so isi-forking would then imply forking.
\end{remark}
As before, we will prove various basic properties of isi-forking.
\begin{proposition}
\thlabel{prop:isi-forking-subobjects}
Let $A, B, C \leq M$ be subobjects. Let $(a, b, c; M)$ and $(a', b', c'; M)$ be two sets of representatives. Then $\gtp(a, b, c; M)$ isi-forks over $c$ if and only if $\gtp(a', b', c'; M)$ isi-forks over $c'$.
\end{proposition}
\begin{proof}
We use the same trick as we did in \thref{prop:long-dividing-subobjects}. Let $f$ be the isomorphism such that $a' = af$. Let $M \to N$ be an extension with $((a_j)_{j \in J}, (d_j)_{j \in J}; N)$ witnessing isi-forking of $\gtp(a, b, c; M)$. Then using \thref{prop:long-dividing-subobjects} we have that $\gtp(a_j f, d_j, c'; N)$ isi-divides over $c'$ for all $j \in J$. We claim that isi-forking of $\gtp(a', b', c'; M)$ is witnessed by $((a_j f)_{j \in J}, (d_j)_{j \in J}; N)$, for which we are now left to check (ii) from \thref{def:isi-forking}.

Let $g$ and $h$ be isomorphisms such that $b' = bg$ and $c' = ch$. Let $N \to N^*$ be an extension with some $(a^*; N^*)$ such that $\gtp(a^*, b', c'; N^*) = \gtp(a', b', c'; N)$. Then
\begin{align*}
\gtp(a^* f^{-1}, b, c; N^*) &= \gtp(a^* f^{-1}, b' g^{-1}, c' h^{-1}; N^*) \\
&= \gtp(a' f^{-1}, b' g^{-1}, c' h^{-1}; N) \\
&= \gtp(a, b, c; N),
\end{align*}
so there is $j \in J$ with $\gtp(a^* f^{-1}, d_j, c; N^*) = \gtp(a_j, d_j, c; N)$. We conclude that $\gtp(a^*, d_j, c'; N^*) = \gtp(a^* f^{-1} f, d_j, c h; N^*) = \gtp(a_j f, d_j, c'; N)$.
\end{proof}
\begin{definition}
\thlabel{def:isi-forking-independence}
For $A, B, C \leq M$ we write $A \ind_C^{\isif,M} B$ if $\gtp(a, b, c; M)$ does not isi-fork for all (equivalently: some) representatives $a, b, c$ of $A, B, C$.
\end{definition}
\begin{proposition}
\thlabel{prop:basic-properties-isi-forking}
Isi-forking satisfies the following: \textsc{Invariance}, \textsc{Monotonicity} on both sides, \textsc{Extension} and \textsc{Base-Monotonicity}.
\end{proposition}
\begin{proof}
The properties \textsc{Invariance} and \textsc{Right-Monotonicity} are direct from the definition. We prove the contraposition of the remaining three.

For \textsc{Left-Monotonicity} suppose that $\gtp(a', b, c; M)$ isi-forks over $c$ and let $(a; M)$ be such that $a'$ factors through $a$. Let $((a'_j)_{j \in J}, (d_j)_{j \in J}; N)$ in some extension $M \to N$ witness the isi-forking. Let $f$ be such that $af = a'$. The following is a set by \thref{fact:galois-type-set}:
\begin{align*}
F = \{\gtp(a^*, d_j, c; N^*) :\,  &N^* \text{ is an extension of } N \text{ and } j \in J \text{ and} \\
&\gtp(a^*, b, c; N^*) = \gtp(a, b, c; M) \text{ and } \\
&\gtp(a^* f, d_j, c; N^*) = \gtp(a_j', d_j, c; N)\}.
\end{align*}
By \textsc{Left-Monotonicity} of isi-dividing, every Galois type in $F$ isi-divides over $c$. By inductively amalgamating things we find one extension $N \to N^*$ with $((a_k)_{k \in K}, (d_k)_{k \in K}; N^*)$  such that every Galois type in $F$ is realised by $(a_k, d_k, c; N^*)$ for some $k \in K$. This then witnesses isi-forking of $\gtp(a, b, c; M)$ over $c$.

For \textsc{Extension} let $(a, b, b', c; M)$ be such that $b$ factors through $b'$ and for every $(a'; N)$ in some extension $M \to N$ with $\gtp(a', b, c; N) = \gtp(a, b, c; M)$ we have that $\gtp(a', b', c; N)$ isi-forks over $c$. So the conclusion of the \textsc{Extension} property for $\gtp(a, b, c; M)$ fails. We have to prove that then $\gtp(a, b, c; M)$ isi-forks over $c$.

For each Galois type in $\Sgtp(\dom(a), \dom(b'), \dom(c))$ we fix some witnesses of isi-forking. By \thref{fact:galois-type-set} and the definition of isi-forking, the following is a set:
\begin{align*}
F = \{&\gtp(a', d, c; N) : N \text{ is an extension of } M \text{ and}\\
&\quad \gtp(a', b, c; N) = \gtp(a, b, c; M) \text{ and } \\
&\quad \gtp(a', d, c; N) \text{ is a fixed witness of isi-forking of } \gtp(a', b', c; N)\}.
\end{align*}
By inductively amalgamating things we find one extension $N \to N^*$ together with $((a_j)_{j \in J}, (d_j)_{j \in J}; N^*)$ such that every Galois type in $F$ is realised by $(a_j, d_j, c; N^*)$ for some $j \in J$. This then witnesses isi-forking of $\gtp(a, b, c; M)$ over $c$.

Finally, for \textsc{Base-Monotonicity} let $(a, b, c, c'; M)$ be such that $\gtp(a, b, c'; M)$ isi-forks over $c'$ and $C \leq C' \leq B$, where $C, C', B$ are the subobjects represented by $c, c', b$ respectively. Let $((a_j)_{j \in J}, (d_j)_{j \in J}; N)$ witness this in some extension $M \to N$. We claim that this also witnesses isi-forking of $\gtp(a, b, c; M)$ over $c$. Indeed, let $a': A \to N'$ for some extension $N \to N'$ be such that $\gtp(a', b, c; N') = \gtp(a, b, c; N)$. We have $C' \leq B$, so $\gtp(a', b, c'; N') = \gtp(a, b, c'; N)$. So there must be some $j \in J$ such that $\gtp(a', d_j, c'; N') = \gtp(a_j, d_j, c'; N)$. As $C \leq C'$ this restricts to $\gtp(a', d_j, c; N') = \gtp(a_j, d_j, c; N)$, which concludes the proof.
\end{proof}
\begin{proposition}
\thlabel{prop:isi-dividing-vs-isi-forking}
For any $A, B, C \leq M$ we always have
\[
A \ind_C^{\isif,M} B \implies A \ind_C^{\isid,M} B.
\]
The converse holds if and only if isi-dividing satisfies \textsc{Right-Monotonicity} and \textsc{Extension}.
\end{proposition}
\begin{proof}
The first implication is just the contrapositive of the trivial statement that isi-dividing implies isi-forking. If the converse of this implication holds, then isi-dividing and isi-forking coincide and so isi-dividing satisfies \textsc{Right-Monotonicity} and \textsc{Extension} by \thref{prop:basic-properties-isi-forking}.

We are left to prove that if isi-dividing satisfies \textsc{Right-Monotonicity} and \textsc{Extension} then isi-forking implies isi-dividing. Suppose for a contradiction that $\gtp(a, b, c; M)$ isi-forks over $c$ but does not isi-divide over $c$. Let $((a_j)_{j \in J}, (d_j)_{j \in J}; N)$ for some extension $M \to N$ witness the isi-forking of $\gtp(a, b, c; M)$. By \textsc{Extension} for isi-dividing we find an extension $N \to N'$ with $(a'; N')$ such that $\gtp(a', b, c; N') = \gtp(a, b, c; N)$ and $\gtp(a', N, c; N')$ does not isi-divide. By isi-forking, there must be $j \in J$ such that $\gtp(a', d_j, c; N')$ isi-divides over $c$ contradicting \textsc{Right-Monotonicity} of isi-dividing.
\end{proof}
When considering NSOP$_1$-theories in first-order logic the useful notion of independence is given by Kim-dividing, see for example \cite{kaplan_kim-independence_2020, dobrowolski_independence_2022}. The idea is to only consider dividing with respect to Morley sequences, that is, with respect to indiscernible nonforking sequences. We adapt that definition to our earlier ideas.
\begin{definition}
\thlabel{def:long-kim-dividing}
We say that $\gtp(a, b, c; M)$ \emph{long Kim-divides} over $c$ if it long divides over $c$ with respect to $\ind_c^{\isif}$-independent sequences. That is, the definition is exactly as long dividing, but we require the sequence $(b_i)_{i < \lambda}$ to be $\ind_c^{\isif}$-independent. We write $A \ind_C^{\lK, M} B$ if $\gtp(a, b, c; M)$ does not long Kim-divide over $c$ for all (equivalently: some) representatives $a, b, c$ of the subobjects $A, B, C$.
\end{definition}
We implicitly used a combination of \thref{prop:long-dividing-subobjects} and \thref{prop:isi-forking-subobjects} to conclude that long Kim-dividing is invariant under taking different representatives of subobjects.

We have defined $\ind_C^\lK$-independence using $\ind_C^\isif$-independent sequences, but these may not exist. For this we define the following axiom, from which the existence of such sequences follows.
\begin{definition}
\thlabel{def:existence-axiom}
Let $(\C, \M)$ be an AECat and let $\B$ be a base class. We say that $(\C, \M)$ satisfies the \emph{$\B$-existence axiom} if $\ind^\isif$ with its base restricted to $\B$ satisfies \textsc{Existence}. That is, for all $A, C \leq M$ with $C \in \B$ we have $A \ind_C^{\isif,M} C$.
\end{definition}
\begin{corollary}
\thlabel{cor:existence-axiom-gives-independent-sequences}
If $(\C, \M)$ satisfies the $\B$-existence axiom then for any $(a, c; M)$ with $\dom(c) \in \B$ and any $\kappa$ there is some extension $M \to N$ containing a $\ind_c^\isif$-independent sequence $(a_i)_{i < \kappa}$ with $\gtp(a_i, c; N) = \gtp(a, c; M)$ for all $i < \kappa$.
\end{corollary}
\begin{proof}
Combine \thref{prop:build-independent-sequence} and \thref{prop:basic-properties-isi-forking}.
\end{proof}
\begin{example}
\thlabel{ex:b-existence}
We discuss some examples of the $\B$-existence axiom. These are either settings where we have the axiom, or where it is natural to assume the axiom.
\begin{enumerate}[label=(\roman*)]
\item For any first-order theory $T$ we have the $\Mod(T)$-existence axiom. This is because any type over a model can be extended to a global invariant type, which can then be used in a standard argument to show that such a type does not isi-fork over $M$ as follows.

Let $q(x) \supseteq \tp(a/M)$ be a global $M$-invariant extension. Let $\alpha \models q$, which then lives in some bigger monster model. Suppose that $\tp(a/M)$ isi-forks. Then there is $d$ such that $r(x, d) = \tp(\alpha / Md)$ isi-divides over $M$. However, we will show that $r(x, d)$ cannot even long divide. So let $(d_i)_{i < \lambda}$ be any infinite sequence in $\tp(d/M)$. Then by $M$-invariance $\tp(\alpha d_i/M) = \tp(\alpha d/M)$ for all $i < \lambda$. So $\bigcup_{i < \lambda} r(x, d_i)$ is consistent, and so we conclude that $r(x, d)$ does not long divide.

\item Analogous to the previous point, for a continuous theory $T$ we have the $\MetMod(T)$-existence axiom.

\item For positive logic something similar to (i) is true, but we need an extra assumption on $T$. We recall these assumptions in more detail in \thref{def:thick-semi-hausdorff}. For now we just summarise what we get from them in terms of the existence axiom. In a semi-Hausdorff positive theory $T$ any type over an e.c.\ model can be extended to a global invariant type (see \cite[Lemma 3.11]{ben-yaacov_thickness_2003}), so the proof in (i) goes through and we have the $\Mod(T)$-existence axiom. In the more general class of thick theories we still have the $\Mod(T)$-existence axiom, but we have to use global Lascar-invariant types instead, see \cite[Lemma 3.11 and Lemma 9.11]{dobrowolski_kim-independence_2022}.

\item For an NSOP$_1$ theory $T$ in first-order logic it is common to assume the existence axiom for forking. It is still an open problem whether or not the existence axiom for forking holds in every NSOP$_1$ theory $T$, but it has been proved in many specific instances, see \cite[Fact 2.14]{dobrowolski_independence_2022}.

If for such $T$ we take $(\C, \M) = (\SubMod(T), \Mod(T))$ then we are very close to having the $\C$-existence axiom. The only difference is that we work with isi-forking, see \thref{rem:isi-forking-vs-forking} for a comparison. In particular, the $\C$-existence axiom implies the existence axiom for forking. Furthermore, if there is a proper class of Ramsey cardinals then isi-forking and forking coincide and so the converse would hold as well. Additionally, it is quite likely that techniques to prove existence for forking also work for isi-forking. For example, in \cite[Remark 2.15]{dobrowolski_independence_2022} it is shown that in the theory of parametrised equivalence relations any type over any set $A$ can be extended to a global $A$-invariant type. Following point (i) we then see that such a type does not isi-fork over $A$.

\item If $(\C, \M)$ is an AECat with a simple independence relation $\ind$ then it will satisfy the $\base(\ind)$-existence axiom. This follows from canonicity, \thref{thm:canonicity-of-simple-independence}, because then $\ind = \ind^{\isif}$ over $\base(\ind)$. This mirrors the fact that simple theories in first-order logic (and even simple thick positive theories, see \cite{ben-yaacov_thickness_2003}) have the existence axiom for forking, see also the previous point.
\end{enumerate}
\end{example}
\begin{remark}
\thlabel{rem:long-kim-dividing-vs-kim-dividing}
The usual definition of Kim-dividing states that a type Kim-divides if it divides with respect to non-forking Morley sequences, see e.g.\ \cite{kaplan_kim-independence_2020, dobrowolski_independence_2022}. To compare this to long Kim-dividing we first note that by \thref{rem:isi-forking-vs-forking} any $\ind^{\isif}$-independent sequence is also a forking-independent sequence, and the converse is true if isi-dividing coincides with dividing. As before, if we assume that there is a proper class of Ramsey cardinals then long Kim-dividing and Kim-dividing coincide, using the same arguments as for long dividing and isi-dividing versus dividing.

If we do not want to assume large cardinals then we can again use canonicity, this time \thref{thm:canonicity-of-nsop1-like-independence}, to see that long Kim-dividing and Kim-dividing coincide in NSOP$_1$-theories where it has been developed.
\end{remark}
\section{Canonicity}
\label{sec:canonicity}
In this section we prove the main results, namely the canonicity theorems for simple (\thref{thm:canonicity-of-simple-independence}) and NSOP$_1$-like (\thref{thm:canonicity-of-nsop1-like-independence}) independence relations. The former is just a slightly improved version of \cite[Theorem 1.1]{kamsma_kim-pillay_2020}. The results in this section up to and including \thref{thm:canonicity-of-simple-independence} are then essentially just the proof of \cite[Theorem 1.1]{kamsma_kim-pillay_2020} cut up in smaller parts. However, we cannot really just refer to that proof again. First of all because we work with slightly different definitions. More importantly, the results here are actually improved versions. Most notably \thref{lem:independence-theorem-lemma} gives a lot more information, where the proof of \cite[Theorem 1.1]{kamsma_kim-pillay_2020} only used the last sentence of that lemma.
\begin{theorem}
\thlabel{thm:isi-dividing-implies-abstract-independence}
Let $(\C, \M)$ be an AECat with AP and let $\ind$ be a basic independence relation that also satisfies \textsc{Club Local Character}. Then $A \ind_C^{\isid, M} B$ implies $A \ind_C^M B$ for any $C \in \base(\ind)$.

If $\ind$ satisfies the same assumptions, except possibly \textsc{Union}, then we still have that $A \ind_C^{\ld, M} B$ implies $A \ind_C^M B$ for any $C \in \base(\ind)$.
\end{theorem}
\begin{proof}
Suppose that $\gtp(a, b, c; M)$ does not isi-divide over $c$. Let $\kappa \geq \Upsilon(A)$ such that $(\C, \M)$ is a $\kappa$-AECat and $\dom(a)$ and $\dom(c)$ are $\kappa$-presentable. By \thref{prop:build-independent-isi-sequence} we find a long enough $\ind_c$-independent isi-sequence $(b_i)_{i < \lambda}$ over $c$ in some $M \to N$ with $\lambda > \kappa$ and $\gtp(b_i, c; N) = \gtp(b, c; M)$ for all $i < \lambda$. Let $(M_i)_{i < \lambda}$ be witnesses of independence. Since $\gtp(a, b, c; M)$ does not isi-divide over $c$ there is $I \subseteq \lambda$ with $|I|  = \kappa$ such that $\gtp(a, b, c; M)$ is consistent for $(b_i)_{i \in I}$. Let $a'$ be a realisation for $(b_i)_{i \in I}$, which for convenience we may assume to be in $N$. By possibly deleting an end segment from $I$ we may assume that $I$ has the order type of $\kappa$. Using \thref{lem:independence-witnesses-presentability} we may assume that each object in the chain $(M_i)_{i \in I}$ is $\kappa$-presentable, where \textsc{Monotonicity} implies that these are still witnesses of independence. Then by chain local character, \thref{lem:chain-local-character}, we find $i_0 \in I$ such that $a' \ind_{M_{i_0}}^N M_I$ where $M_I = \colim_{i \in I} M_i$. By \textsc{Monotonicity} and \textsc{Symmetry} we then have
\[
b_{i_0} \ind_{M_{i_0}}^N a'.
\]
We also have
\[
b_{i_0} \ind_c^N M_{i_0}.
\]
So by \textsc{Transitivity} we have $b_{i_0} \ind_c^N a'$ and the result follows by \textsc{Symmetry} and because $\gtp(a', b_{i_0}, c; N) = \gtp(a, b, c; M)$.

For the final claim we just note that if we do not have \textsc{Union} we can still apply \thref{prop:build-independent-sequence} instead of \thref{prop:build-independent-isi-sequence} to get an arbitrarily long $\ind_c$-independent sequence. It might just not be an isi-sequence. Then the rest of the proof goes through as written.
\end{proof}
The following lemma generalises the \textsc{Independence Theorem} property to independent sequences of any length. The original \textsc{Independence Theorem} can roughly be viewed as just considering an independent sequence of length two.
\begin{lemma}[Generalised independence theorem]
\thlabel{lem:independence-theorem-lemma}
Suppose that $\ind$ is a basic independence relation satisfying \textsc{Independence Theorem}. Let $\delta$ be any (possibly finite) ordinal. Suppose we have $(a, b, c; N)$ such that $a \ind_c^N b$ and a $\ind_c$-independent sequence $(b_i)_{i < \delta}$ in $N$ with $\Lgtp(b_i / c; N) = \Lgtp(b / c; N)$ for all $i < \delta$. Then there is an extension $N \to N'$ with $(a'; N')$ such that $a' \ind_c^{N'} N$ and $\Lgtp(a', b_i / c; N') = \Lgtp(a, b / c; N')$ for all $i < \delta$.

In particular, $\gtp(a, b, c; N)$ is consistent for $(b_i)_{i < \delta}$.
\end{lemma}
\begin{proof}
Let $(M_i)_{i < \delta}$ be witnesses of independence for $(b_i)_{i < \delta}$, which we may assume to be subobjects of $N$. We will add one more link $M_\delta$ to the chain. If $\delta$ is a limit ordinal we set $M_\delta = \colim_{i < \delta} M_i$. If $\delta$ is a successor ordinal we set $M_\delta = N$.

We will by induction construct a chain $(N_i)_{i \leq \delta}$ with $N_0$ extending $N$, together with extensions $\{m_i': M_i \to N_i\}_{i \leq \delta}$ and an arrow $(a''; N_0)$ such that $m_0' = m_0$ and $\Lgtp(a'' / c; N_0) = \Lgtp(a / c; N_0)$ while at stage $i$ we have:
\begin{enumerate}[label=(\roman*)]
\item the extensions $\{m_j': M_j \to N_j\}_{j \leq i}$ are natural in the sense that
\[
\begin{tikzcd}[sep=small]
N_j \arrow[r]                   & N_i                    \\
M_j \arrow[r] \arrow[u, "m_j'"] & M_i \arrow[u, "m_i'"']
\end{tikzcd}
\]
commutes for all $j \leq i$;
\item if $i$ is a successor, say $i = j + 1$, then $\Lgtp(a'', b_j' / c; N_i) = \Lgtp(a, b / c; N_i)$, where $b_j'$ is the composition $B \xrightarrow{b_j} M_i \xrightarrow{m_i'} N_i$;
\item $a'' \ind_c^{N_i} m_i'$.
\end{enumerate}
\underline{Base case.} By \textsc{Existence} we have $a \ind_c^N c$, so we can apply strong extension (\thref{cor:strong-extension-dual}) to find $N \to N_0$ and $(a''; N_0)$ with $a'' \ind_c^{N_0} M_0$ and $\Lgtp(a'' / c; N_0) = \Lgtp(a / c; N_0)$.

\vspace{\baselineskip}\noindent
\underline{\emph{Successor step.}} Suppose we have constructed $N_i$ and $m_i'$. As $M_i$ is an amalgamation base we have $\gtp(m_i; N_i) = \gtp(m_i'; N_i)$. By (i) we have that $m_0'$ factors through $m_i'$ in the same way that $m_0$ factors through $m_i$, so $\gtp(m_i', m_0'; N_i) = \gtp(m_i, m_0; N_i)$. Since $m_0' = m_0$ have $\gtp(m_i', m_0; N_i) = \gtp(m_i, m_0; N_i)$, so $\Lgtp(m_i' / c; N_i) = \Lgtp(m_i / c; N_i)$. We thus find $(a^*, m_{i+1}^*; N^*)$ for some $N_i \to N^*$ such that $\Lgtp(m_{i+1}^*, m_i' / c; N^*) = \Lgtp(m_{i+1}, m_i / c; N^*)$ and $\Lgtp(a^*, b_i^* / c; N^*) = \Lgtp(a, b / c; N^*)$, where $b_i^*$ is given by $B \xrightarrow{b_i} M_{i+1} \xrightarrow{m_{i+1}^*} N^*$. For this last construction we used that $\Lgtp(b_i / c; N) = \Lgtp(b / c; N)$ and that $b_i$ factors through $m_{i+1}$. Then $a'' \ind_c^{N^*} m_i'$, $a^* \ind_c^{N^*} b_i^*$ and $b_i^* \ind_c^{N^*} m_i'$. So by \textsc{Independence Theorem} we find $N^* \to N_{i+1}$ and $(a^{**}; N_{i+1})$ with $\Lgtp(a^{**}, b_i^* / c; N_{i+1}) = \Lgtp(a^*, b_i^* / c; N_{i+1})$, $\Lgtp(a^{**}, m_i' / c; N_{i+1}) = \Lgtp(a'', m_i' / c; N_{i+1})$ and $a^{**} \ind_c^{N_{i+1}} N^*$. By \textsc{Monotonicity} we get $a^{**} \ind_c^{N_{i+1}} m_{i+1}^*$. Using $\Lgtp(a^{**}, m_i' / c; N_{i+1}) = \Lgtp(a'', m_i' / c; N_{i+1})$ we find $m_{i+1}': M_{i+1} \to N_{i+1}$ (after replacing $N_{i+1}$ by an extension) such that $\Lgtp(a^{**}, m_{i+1}^*, m_i' / c; N_{i+1}) = \Lgtp(a'', m_{i+1}', m_i' / c; N_{i+1})$. We verify the induction hypothesis:
\begin{enumerate}[label=(\roman*)]
\item we have by construction that $\gtp(m_{i+1}', m_i'; N_{i+1}) = \gtp(m_{i+1}^*, m_i'; N_{i+1}) = \gtp(m_{i+1}, m_i: N_{i+1})$, so $m_i'$ factors through $m_{i+1}'$ in the same way that $m_i$ factors through $m_{i+1}$ by \thref{fact:galois-type-factorisation}, and naturality follows;
\item using $\Lgtp(a^{**}, m_{i+1}^* / c; N_{i+1}) = \Lgtp(a'', m_{i+1}' / c; N_{i+1})$ and that $b_i^*$ and $b_i'$ are $B \xrightarrow{b_i} M_{i+1} \xrightarrow{m_{i+1}^*} N_{i+1}$ and $B \xrightarrow{b_i} M_{i+1} \xrightarrow{m_{i+1}'} N_{i+1}$ respectively by definition, we have $\Lgtp(a'', b_i' / c; N_{i+1}) = \Lgtp(a^{**}, b_i^*/ c; N_{i+1}) = \Lgtp(a^*, b_i^* / c; N_{i+1}) = \Lgtp(a, b / c; N_{i+1})$;
\item by $a^{**} \ind_c^{N_{i+1}} m_{i+1}^*$ and \textsc{Invariance}.
\end{enumerate}

\vspace{\baselineskip}\noindent
\underline{\emph{Limit step.}} For limit $\ell$ let $N_\ell = \colim_{i < \ell} N_i$. By (i) from the induction hypothesis the arrows $m_i'$ composed with the coprojections $N_i \to N_\ell$ form a cocone on $(M_i)_{i < \ell}$. By continuity $M_\ell = \colim_{i < \ell} M_i$, so there is a universal arrow $m_\ell': M_\ell \to N_\ell$. This directly establishes (i). Property (ii) is vacuous. Property (iii) follows from the induction hypothesis and \textsc{Union}.

\vspace{\baselineskip}\noindent Having finished the inductive construction, we have two arrows $M_\delta \to N_\delta$, namely $m_\delta: M_\delta \to N \to N_\delta$ and the $m_\delta'$ we just constructed. By (i) from the induction hypothesis we have $\gtp(m_\delta, m_0; N_\delta) = \gtp(m_\delta', m_0; N_\delta)$. So we find an extension $N_\delta \to N'$ and some $(a'; N')$ such that $\gtp(a', m_\delta, m_0; N') = \gtp(a'', m_\delta', m_0; N_\delta)$. Using that $c$ factors through $m_0$ and (ii) from the induction hypothesis, we find that for any $i < \delta$ we have $\Lgtp(a', b_i / c; N') = \Lgtp(a'', b_i' / c; N') = \Lgtp(a, b / c; N')$. By (iii) from the induction hypothesis we also have $a' \ind_c^{N'} M_\delta$. So if $\delta$ was a successor ordinal we had $M_\delta = N$ and we are done. Otherwise we can just apply \textsc{Extension} and relabel things to get $a' \ind_c^{N'} N$.

The final claim follows because $a'$ is a realisation of $\gtp(a, b, c; N)$ for $(b_i)_{i < \delta}$.
\end{proof}
\begin{remark}
\thlabel{rem:independence-theorem-lemma-no-lgtp}
In the context of \thref{lem:independence-theorem-lemma} if $C$ is a model then there is no need to concern ourselves with Lascar strong Galois types. That is, the proof as written then goes through if we replace ``Lascar strong Galois type'' by just ``Galois type'' everywhere. We also only apply \textsc{Independence Theorem} with $C$ in the base. So if $C$ is a model then it would be enough to just have \textsc{Independence Theorem} over models. Or equivalently, to have \textsc{3-amalgamation}, see \thref{thm:independence-theorem-vs-3-amalgamation}.
\end{remark}
The following is a slightly improved version of \cite[Theorem 1.1]{kamsma_kim-pillay_2020}. The improvement is in the fact that we can restrict our independence relation to a base class and the fact that we also get $\ind = \ind^{\isif}$.
\begin{repeated-theorem}[\thref{thm:canonicity-of-simple-independence}]
Let $(\C, \M)$ be an AECat with the amalgamation property, and suppose that $\ind$ is a simple independence relation. Then $\ind = \ind^{\isid} = \ind^{\isif}$ over $\base(\ind)$.
\end{repeated-theorem}
\begin{proof}
The implication $\ind^{\isid} \implies \ind$ is already given by \thref{thm:isi-dividing-implies-abstract-independence}. For the converse we will assume that $A \ind_C^M B$ and we will prove that $A \ind_C^{\isid, M} B$. Pick some representatives $a, b, c$ of $A, B, C$. Let $\mu$ be such that $(\C, \M)$ is a $\mu$-AECat and let $\lambda > \Upsilon(B) + \mu$. Let $(b_i)_{i < \lambda}$ be an isi-sequence over $c$ in some $M \to N$, with chain of initial segments $(M_i)_{i < \lambda}$ and $\gtp(b_i, c; N) = \gtp(b, c; M)$ for all $i < \lambda$. Let $\Upsilon(B) + \mu \leq \kappa < \lambda$. By \thref{lem:independence-witnesses-presentability} we may assume that $M_i$ is $\kappa$-presentable for all $i < \kappa$. We can thus apply chain local character, \thref{lem:chain-local-character}, to find $i_0 < \kappa$ such that $b_\kappa \ind_{M_{i_0}}^N M_\kappa$. We will aim to show that $\gtp(a, b, c; M)$ is consistent for $(b_i)_{i_0 \leq i < \kappa}$. We use $\gtp(b_{i_0}, c; N) = \gtp(b, c; M)$ to find a common extension $M \to N' \leftarrow N$ where $b = b_{i_0}$ as arrows into $N'$. By applying \textsc{Extension} to the assumption $a \ind_c^M b$ we then find $(a'; N')$ (possibly after replacing $N'$ by an extension) such that $a' \ind_c^{N'} N$ and $\gtp(a', b, c; N') = \gtp(a, b, c; M)$. Then by \textsc{Base-Monotonicity} and \textsc{Monotonicity} we find $a' \ind_{M_{i_0}}^{N'} b$. For any $i_0 \leq i < \kappa$ we have $\gtp(b_i, m_i, m_{i_0}; N') = \gtp(b_\kappa, m_i, m_{i_0}; N')$ because $(b_i)_{i < \lambda}$ is an isi-sequence. So by \textsc{Monotonicity} and \textsc{Invariance} and the earlier fact that $b_\kappa \ind_{M_{i_0}}^N M_\kappa$, we find
\[
b_i \ind_{M_{i_0}}^{N'} M_i
\]
for all $i_0 \leq i < \kappa$. So $(b_i)_{i_0 \leq i < \kappa}$ is a $\ind_{M_{i_0}}$-independent sequence. We can thus apply the generalised independence theorem, \thref{lem:independence-theorem-lemma}, to conclude that $\gtp(a', b, c; N') = \gtp(a, b, c; M)$ is indeed consistent for $(b_i)_{i_0 \leq i < \kappa}$. As $\kappa$ was arbitrarily large below $\lambda$, $\lambda$ itself was arbitarily large and $(b_i)_{i_0 \leq i < \kappa}$ is a subsequence of an arbitrary isi-sequence of length $\lambda$ we conclude that indeed $A \ind_C^{\isid, M} B$.

Finally, the claim $\ind^{\isid} = \ind^{\isif}$ follows from \thref{prop:isi-dividing-vs-isi-forking} because $\ind^{\isid} = \ind$ has \textsc{Extension} and \textsc{Right-Monotonicity}.
\end{proof}
\begin{remark}
\thlabel{rem:3-amal-enough-for-simple-canonicity-part2}
For the canonicity theorem for simple independence relations, \thref{thm:canonicity-of-simple-independence}, we only need \textsc{3-amalgamation}. Or equivalently, by \thref{thm:independence-theorem-vs-3-amalgamation}, \textsc{Independence Theorem} over models. Even if $\base(\ind)$ is more than just $\M$, e.g.\ $\base(\ind) = \C$. In the proof of \thref{thm:canonicity-of-simple-independence} we only applied the \textsc{Independence Theorem} indirectly through \thref{lem:independence-theorem-lemma}. The base, i.e.\ $C$ in that lemma, is by construction always a model. So by \thref{rem:independence-theorem-lemma-no-lgtp} it would be enough to only assume \textsc{3-amalgamation} instead of \textsc{Independence Theorem}.
\end{remark}
\begin{repeated-theorem}[\thref{thm:canonicity-of-nsop1-like-independence}]
Let $(\C, \M)$ be an AECat with the amalgamation property and let $\B$ be some base class. Suppose that $(\C, \M)$ satisfies the $\B$-existence axiom and suppose that there is an NSOP$_1$-like independence relation $\ind$ over $\B$. Then $\ind = \ind^{\lK}$ over $\B$.
\end{repeated-theorem}
\begin{proof}
Suppose that $A \ind_C^M B$ with $C \in \B$ and pick some representatives $a, b, c$ of $A, B, C$. There is a bound $\mu$ on the cardinality of the set of Lascar strong Galois types compatible with $(b, c; M)$, see \thref{prop:lascar-strong-galois-types-bounded}. Let $\lambda > \mu$ and let $(b_i)_{i < \lambda}$ be a $\ind_c^{\isif}$-independent sequence in some $M \to N$ with $\gtp(b_i, c; N) = \gtp(b, c; M)$ for all $i < \lambda$. Then $(b_i)_{i < \lambda}$ is also $\ind_c$-independent, by \thref{thm:isi-dividing-implies-abstract-independence} and \thref{prop:isi-dividing-vs-isi-forking}. We have to show that for every $\kappa < \lambda$ there is $I \subseteq \lambda$ with $|I| = \kappa$ such that $\gtp(a, b, c; M)$ is consistent for $(b_i)_{i \in I}$. So let $\kappa < \lambda$. Then by the choice of $\mu$ and $\lambda$ there must be some $I \subseteq \lambda$ with $|I| = \kappa$ such that $\Lgtp(b_i / c; N) = \Lgtp(b_j / c; N)$ for all $i,j \in I$. Let $i_0$ be the least element of $I$. Let $(a'; N')$ for some extension $N \to N'$ be such that $\gtp(a', b_{i_0}, c; N') = \gtp(a, b, c; M)$. Then we can apply the generalised independence theorem, \thref{lem:independence-theorem-lemma}, to see that $\gtp(a', b_{i_0}, c; N') = \gtp(a, b, c; M)$ is consistent for $(b_i)_{i \in I}$. We conclude that indeed $A \ind_C^{\lK, M} B$.

For the other direction, suppose that $A \ind_C^{\lK, M} B$ with $C \in \B$. Let $\kappa \geq \Upsilon(A)$ be such that $(\C, \M)$ is a $\kappa$-AECat and $A$ and $C$ are $\kappa$-presentable. Let $(b_i)_{i < \lambda}$ be a long enough $\ind_c^{\isif}$-independent sequence in some extension $M \to N$, with $\lambda > \kappa$, witnesses of independence $(M_i)_{i < \lambda}$ and $\gtp(b_i, c; N) = \gtp(b, c; M)$ for all $i < \lambda$. Such a sequence exists by \thref{cor:existence-axiom-gives-independent-sequences}, because we assumed the $\B$-existence axiom. By \thref{thm:isi-dividing-implies-abstract-independence} and \thref{prop:isi-dividing-vs-isi-forking} this is also a $\ind_c$-independent sequence. By definition of long Kim-dividing there is $I \subseteq \lambda$ with $|I| = \kappa$ such that $\gtp(a, b, c; N)$ is consistent for $(b_i)_{i \in I}$. Let $a'$ be a realisation for this (we may assume $a'$ is an arrow into $N$). By possibly deleting an end segment from $I$ we may assume that $I$ has the order type of $\kappa$. Using \thref{lem:independence-witnesses-presentability} we may assume that each object in the chain $(M_i)_{i \in I}$ is $\kappa$-presentable, where \textsc{Monotonicity} guarantees that these are still witnesses of independence. Then by chain local character, \thref{lem:chain-local-character}, we find $i_0 \in I$ such that $a' \ind_{M_{i_0}}^N M_I$ where $M_I = \colim_{i \in I} M_i$. So by \textsc{Monotonicity} and \textsc{Symmetry} we have
\[
b_{i_0} \ind_{M_{i_0}}^N a'.
\]
Furthermore, we have
\[
b_{i_0} \ind_C^N M_{i_0}.
\]
So by \textsc{Transitivity} we have $b_{i_0} \ind_C^N a'$.  The result then follows by \textsc{Symmetry} and the fact that $\gtp(a', b_{i_0}, c; N) = \gtp(a, b, c; M)$.
\end{proof}
By definition any stable independence relation is also simple, and any simple independence relation is also NSOP$_1$-like. The canonicity theorems then tell us that these are indeed unique in a given AECat with AP and with what notion of dividing they coincide. We make this precise in the following theorem.
\begin{repeated-theorem}[\thref{thm:independence-hierarchy}]
Let $(\C, \M)$ be an AECat with the amalgamation property and suppose that $\ind$ is a stable or a simple independence relation in $(\C, \M)$. Suppose furthermore that $\ind^*$ is an NSOP$_1$-like independence relation in $(\C, \M)$ with $\base(\ind) = \base(\ind^*)$. Then
\[
\ind = \ind^*  = \ind^{\isid} = \ind^{\isif} = \ind^{\lK}.
\]
\end{repeated-theorem}
\begin{proof}
This follows directly from \thref{thm:canonicity-of-simple-independence} and \thref{thm:canonicity-of-nsop1-like-independence}. To apply the latter we need the $\base(\ind)$-existence axiom. This is automatic, as $\ind = \ind^{\isif}$ over $\base(\ind)$ by \thref{thm:canonicity-of-simple-independence} and we have \textsc{Existence} by assumption.
\end{proof}
\begin{remark}
\thlabel{rem:recover-part-of-stability-hierarchy}
We can classify AECats based on the existence of certain independence relations, just as we can classify theories in first-order logic in that way. For example, suppose that we have an AECat $(\C, \M)$ with AP and an NSOP$_1$-like independence relation $\ind$ where \textsc{Base-Monotonicity} fails. Then we can never find a simple independence relation in $(\C, \M)$ (with the same base class). Because if we would have such a simple independence relation $\ind'$ then by \thref{thm:independence-hierarchy} we would have $\ind = \ind'$, but that is impossible because a simple independence relation must satisfy \textsc{Base-Monotonicity}. So we can classify $(\C, \M)$ as NSOP$_1$, but non-simple.
\end{remark}
We close out this section by discussing how this work extends and brings together previously known results in the settings of first-order, positive and continuous logic. We also describe precisely how to apply the canonicity theorems in these more concrete settings.

We first recall some useful terminology for positive logic from \cite{ben-yaacov_positive_2003, ben-yaacov_thickness_2003}.
\begin{definition}
\thlabel{def:thick-semi-hausdorff}
Let $T$ be a positive theory. We call $T$:
\begin{itemize}
\item \emph{semi-Hausdorff} if equality of types is type-definable;
\item \emph{thick} if being an indiscernible sequence is type-definable.
\end{itemize}
\end{definition}
It quickly follows that any first-order theory is semi-Hausdorff as a positive theory and that any semi-Hausdorff theory is thick. So whenever we mention semi-Hausdorff or thick theories in the examples below this automatically includes the first-order setting.

In a semi-Hausdorff theory $T$ we have that having the same type over an e.c.\ model implies having the same Lascar strong type (see \cite[Proposition 3.13]{ben-yaacov_thickness_2003}), just as we have for first-order theories. If $T$ is thick this is no longer generally true, see \cite[Subsection 10.1]{dobrowolski_kim-independence_2022}. In \cite[Lemma 2.20]{dobrowolski_kim-independence_2022} this is solved by considering $\lambda_T$-saturated models, where $\lambda_T = \beth_{(2^{|T|})^+ }$. The problem for AECats is then that the full subcategory of $\lambda_T$-saturated models in $\Mod(T)$ is not closed under directed colimits. We solve this with the following notion.
\begin{definition}
\thlabel{def:finitely-lambda-saturated}
We call an e.c.\ model $M$ of some positive theory $T$ \emph{finitely $\lambda$-saturated} if for every finite tuple $a \in M$ there is a $\lambda$-saturated e.c.\ model $M_0 \subseteq M$ with $a \in M_0$.
\end{definition}
Clearly any $\lambda$-saturated model is also finitely $\lambda$-saturated. The point is that the full subcategory of finitely $\lambda$-saturated models in $\Mod(T)$ is then closed under directed colimits for any $\lambda$. At the same time this notion is strong enough to give us the following fact.
\begin{fact}[{\cite[Proposition 2.39]{kamsma_independence_2021}}]
\thlabel{fact:finitely-lambda-saturated-gives-lascar-type}
In a thick positive theory $T$ having the same Lascar strong type over some parameter set $C$ is the transitive closure of having the same type over finitely $\lambda_T$-saturated models containing $C$.
\end{fact}
\begin{example}
\thlabel{ex:simple-canonicity-positive-logic}
Let $T$ be a thick theory. Let $\C$ be either $\SubMod(T)$ or $\Mod(T)$. Following \thref{fact:finitely-lambda-saturated-gives-lascar-type} we take $\M$ to be the category of finitely $\lambda_T$-saturated models, so that Lascar strong types and Lascar strong Galois types coincide. If $T$ is semi-Hausdorff we can instead just take $\M = \Mod(T)$.

If $T$ is stable or simple then there is respectively a stable or simple independence relation $\ind$ in $(\C, \M)$ with $\base(\ind) = \C$. This follows from a combination of \cite{ben-yaacov_simplicity_2003, ben-yaacov_thickness_2003}. So \thref{thm:canonicity-of-simple-independence} applies.

The \textsc{Stationarity} property in a stable theory follows from \cite[Theorem 2.8]{ben-yaacov_simplicity_2003}. In their statement the base model $M$ is assumed to be $|T|^+$-saturated, which they need for only two reasons. The first reason is that types over $M$ should coincide with Lascar strong types over $M$, but by our choice of $\M$ and the thickness assumption this happens for all $M \in \M$. The second reason is that types over $M$ should be what they call extendible, but in a simple thick theory every type is extendible, see \cite[Theorem 1.15]{ben-yaacov_thickness_2003}.
\end{example}
\begin{example}
\thlabel{ex:nsop1-canonicity-positive-logic}
Let $(\C, \M)$ be an AECat based on some semi-Hausdorff or thick theory $T$ as in \thref{ex:simple-canonicity-positive-logic}. By \thref{ex:b-existence}(iii) we have the $\Mod(T)$-existence axiom. If $T$ is NSOP$_1$ then it has an NSOP$_1$-like independence relation $\ind$ with $\base(\ind) = \Mod(T)$. For first-order logic this was proved in \cite{kaplan_kim-independence_2020, kaplan_local_2019, kaplan_transitivity_2021}, which was extended to thick positive theories in \cite{dobrowolski_kim-independence_2022}. So \thref{thm:canonicity-of-nsop1-like-independence} applies.

Here we had to restrict the base class to e.c.\ models, simply because Kim-independence in positive logic has only been developed over e.c.\ models. For theories in first-order logic Kim-independence has been extended to arbitrary base sets, see \cite{chernikov_transitivity_2020, dobrowolski_independence_2022}. To make this work we need to assume the existence axiom, see also \thref{ex:b-existence}(iv). So let $T$ be an NSOP$_1$ theory in first-order logic and set $(\C, \M) = (\SubMod(T), \Mod(T))$, which we assume to satisfy the $\C$-existence axiom. Then the aforementioned sources show that there is an NSOP$_1$-like independence relation $\ind$ with $\base(\ind) = \C$ and so \thref{thm:canonicity-of-nsop1-like-independence} applies.

Finally we note that there is a Kim-Pillay style theorem in \cite[Theorem 5.1]{chernikov_transitivity_2020} for Kim-independence over arbitrary sets. They still rely on a syntactical property ``strong finite character'', which could be replaced by just ``finite character''. \thref{thm:canonicity-of-nsop1-like-independence} gives us just the canonicity part. To conclude that a theory with such an independence relation is NSOP$_1$, without using strong finite character, we can restrict ourselves to work over models and use the proof from \cite[Theorem 9.1]{dobrowolski_kim-independence_2022}.
\end{example}
\begin{example}
\thlabel{ex:simple-stable-canonicity-continuous-logic}
Let $T$ be a continuous theory, in the sense of \cite{ben-yaacov_model_2008}. Let $\C$ be either $\SubMetMod(T)$ or $\MetMod(T)$, and let $\M$ be $\MetMod(T)$. If $T$ is stable or simple then there is respectively a stable or simple independence relation $\ind$ in $(\C, \M)$ with $\base(\ind) = \C$. Every continuous theory is in particular a Hausdorff compact abstract theory, and so the machinery of \cite{ben-yaacov_simplicity_2003, ben-yaacov_thickness_2003} applies. This shows we can indeed find an appropriate independence relation in any simple or stable continuous theory. There is also \cite[Section 14]{ben-yaacov_model_2008} for a further discussion about stability specifically in continuous theories. So \thref{thm:canonicity-of-simple-independence} applies.

In \cite{ben-yaacov_model_2008} some examples of stable continuous theories and their corresponding independence relations are given, including Hilbert spaces and atomless probability spaces.
\end{example}
\begin{example}
\thlabel{ex:continuous-nsop1}
In this example we consider the continuous theory $T_N$ of Hilbert spaces with a distance function to a random subset, as studied in \cite{berenstein_hilbert_2018}. They prove that this theory has TP$_2$ and thus cannot be simple. They also define an independence relation $\ind^*$ over arbitrary sets that has all the properties of an NSOP$_1$-like independence relation. Except that they do not prove the full \textsc{Independence Theorem}, but enough for \textsc{3-amalgamation} (i.e.\ over models, see \thref{thm:independence-theorem-vs-3-amalgamation}). So setting $\C = \SubMetMod(T_N)$ and $\M = \MetMod(T_N)$, and taking $\base(\ind^*) = \MetMod(T)$, we have that $\ind^*$ is an NSOP$_1$-like independence relation in $(\C, \M)$. By \thref{ex:b-existence}(ii) we also have the $\MetMod(T)$-existence axiom. So \thref{thm:canonicity-of-nsop1-like-independence} applies.
\end{example}
\section{More on Lascar strong Galois types}
\label{sec:lgtp-vs-bounded-gtp}
In this section we will show that in the presence of a nice enough independence relation there are some equivalent definitions of Lascar strong Galois types, matching those we classically have for Lascar strong types. To place this all in context we recall the relevant equivalent definitions of Lascar strong types in first-order logic (see e.g.\ \cite[Proposition 3.1.5]{kim_simplicity_2014}).
\begin{definition}
\thlabel{def:lascar-strong-type}
Let $a$ and $b$ be tuples in some monster model of a first-order theory and let $C$ be some parameter set. We say that $a$ and $b$ have the same \emph{Lascar strong type over $C$} if the following equivalent conditions hold.
\begin{enumerate}[label=(\roman*)]
\item There are $a = a_0, a_1, \ldots, a_n = b$ and models $M_1, \ldots, M_n$, each containing $C$, such that $\tp(a_i / M_{i+1}) = \tp(a_{i+1} / M_{i+1})$ for all $0 \leq i < n$.
\item We have $a \sim b$ for any bounded $C$-invariant equivalence relation $\sim$.
\item There are $a = a_0, a_1, \ldots, a_n = b$ such that $a_i$ and $a_{i+1}$ are on a $C$-indiscernible sequence for all $0 \leq i < n$. In this case we say that $a$ and $b$ have \emph{Lascar distance} at most $n$ (over $C$).
\end{enumerate}
\end{definition}

It is well known that Lascar strong types heavily interact with independence relations in first-order logic. For example, independence relations can be used to show that having the same Lascar strong type is type-definable in any simple theory by showing that the Lascar distance within a Lascar strong type is at most $2$, see \cite[Proposition 5.1.11]{kim_simplicity_2014}. The same technique applies to any NSOP$_1$ theory in first-order logic that satisfies the existence axiom \cite[Corollary 5.9]{dobrowolski_independence_2022}. We essentially adapt this technique in this section, while at the same time using independence relations to build what we call ``strongly 2-indiscernible'' sequences (\thref{def:2-indiscernible-sequence}), which take the role of the usual indiscernible sequences.

Throughout this section we will work with single arrows $a, b, c$ and objects $A$ and $C$, where $\dom(a) = \dom(b) = A$ and $\dom(c) = C$. This leads to cleaner notation and when working with independence relations we can only work with single arrows anyway (i.e.\ the sides and base of an independence relation do not allow tuples of arrows in our definition). However, it is not too difficult to extend the main result of this section (\thref{thm:lascar-strong-type-equivalent-formulations}) to arbitrary tuples, see \thref{rem:lgtp-equivalences-tuples}.
\begin{definition}
\thlabel{def:bounded-invariant-equivalence-relation}
Let $(\C, \M)$ be an AECat with AP and fix some objects $A$ and $C$. Suppose that for each $M$ and each $c: C \to M$ we are given an equivalence relation $\equiv_{c,M}$ on $\Hom(A, M)$. Then we say that the family $\equiv$ is an \emph{equivalence relation over $C$}.

We call $\equiv$ a \emph{bounded equivalence relation} if there is $\lambda$ such that $\equiv_{c,M}$ has at most $\lambda$ many equivalence classes for any $M$ and $c: C \to M$.

We call $\equiv$ an \emph{invariant equivalence relation} if it is invariant under equality of Galois types over $C$. That is, if $\gtp(a, b, c; M) = \gtp(a', b', c'; M')$ then we have that $a \equiv_{c, M} b$ if and only if $a' \equiv_{c', M'} b'$.
\end{definition}
\begin{convention}
\thlabel{conv:invariant-equivalence-relation}
We will only deal with invariant equivalence relations. To further simplify the notation we will drop the $M$ from the notation. So we write $a \equiv_c b$ instead of $a \equiv_{c, M} b$. Because of invariance it does not matter if we consider $a$ and $b$ as arrows into $M$ or as arrows into an extension of $M$.
\end{convention}
\begin{example}
\thlabel{ex:bounded-invariant-equivalence-relations}
We give some familiar examples.
\begin{enumerate}[label=(\roman*)]
\item Taking just equality as equivalence relation is an equivalence relation over any $C$. This relation is invariant, but generally not bounded because $\Hom(A, M)$ may become arbitrarily large.
\item The trivial equivalence where everything is equivalent is a bounded invariant equivalence relation over any $C$.
\item Having the same Galois type is a bounded invariant equivalence relation over any $C$. That is, we define $\equiv$ as $a \equiv_c b$ if and only if $\gtp(a, c; M) = \gtp(b, c; M)$. Clearly $\equiv$ is invariant, and by \thref{fact:galois-type-set} it is bounded.
\item Having the same Lascar strong Galois type is a bounded invariant equivalence relation. Similar to the previous point we define $\equiv$ as $a \equiv_c b$ if and only if $\Lgtp(a / c; M) = \Lgtp(a' / c; M)$. This is invariant by \thref{prop:lascar-strong-type-invariant-under-galois-type} and bounded by \thref{prop:lascar-strong-galois-types-bounded}.
\item In any (positive or first-order) theory $T$, any hyperimaginary yields an invariant equivalence relation. That is, if $E(x, y)$ is a set of formulas that defines an equivalence relation modulo $T$ then we can define an invariant equivalence relation $\equiv^E$ over $\emptyset$ as follows: for tuples $a, b \in M$ we set $a \equiv^E b$ iff $M \models E(a, b)$.
\end{enumerate}
\end{example}
Usually bounded invariant equivalence relations are linked to Lascar strong types using indiscernible sequences. This requires some compactness, which we generally do not have. To solve this we will adapt the idea of strongly indiscernible sequences from \cite{hyttinen_simplicity_2006}.
\begin{definition}
\thlabel{def:2-indiscernible-sequence}
We call sequence $(a_i)_{i < \kappa}$ in some $M$ \emph{2-$c$-indiscernible} if for any $i_1 < i_2 < \kappa$ and any $j_1 < j_2 < \kappa$ we have $\gtp(a_{i_1}, c; M) = \gtp(a_{i_2}, c; M)$ and $\gtp(a_{i_1}, a_{i_2}, c; M) = \gtp(a_{j_1}, a_{j_2}, c; M)$. We call such a sequence \emph{strongly 2-$c$-indiscernible} if it can be extended to a 2-$c$-indiscernible sequence (possibly in some extension model) of arbitrary length.

Given an independence relation $\ind$ we define a \emph{2-$\ind_c$-Morley sequence} to be a 2-$c$-indiscernible sequence that is also $\ind_c$-independent. Such a sequence is called a \emph{strong 2-$\ind_c$-Morley sequence} if it can be extended to a 2-$\ind_c$-Morley sequence (possibly in some extension model) of arbitrary length.
\end{definition}
\begin{definition}
\thlabel{def:lascar-equivalence-relations}
For $(a, b, c; M)$ we write $a \sim^2_c b$ if $a$ and $b$ are on some strongly 2-$c$-indiscernible sequence. We write $\equiv^2_c$ for the transitive closure of $\sim^2_c$. Similarly, given an independence relation $\ind$, we write $a \sim^{\ind}_c b$ if $a$ and $b$ are on some strong 2-$\ind_c$-Morley sequence and $\equiv^{\ind}_c$ for its transitive closure. Finally, we write $a \equiv^B_c b$ if $a$ and $b$ are equivalent for every bounded invariant relation over $c$.
\end{definition}
One easily verifies that $\equiv^2$ and $\equiv^B$ are equivalence relations over any $C$. For $\equiv^{\ind}$ we may generally not have reflexivity, but we will have that in the situations we are interested in. In particular, \thref{lem:same-type-over-model-implies-strong-morley-equivalent} shows that $\equiv^{\ind}$ has reflexivity over models.
\begin{theorem}
\thlabel{thm:lascar-strong-type-equivalent-formulations}
Let $(\C, \M)$ be an AECat with AP, and suppose that $\ind$ is a basic independence relation that also satisfies \textsc{3-amalgamation}. Then the following are equivalent for any $(a, b, c; M)$:
\begin{enumerate}[label=(\roman*)]
\item $\Lgtp(a/c; M) = \Lgtp(b/c; M)$;
\item $a \equiv^B_c b$, so $a$ and $b$ are equivalent under every bounded invariant equivalence relation over $C$;
\item $a \equiv^2_c b$, so $a$ and $b$ can be connected by strongly 2-$c$-indiscernible sequences.
\end{enumerate}
\end{theorem}
Comparing the above statement to \thref{def:lascar-strong-type} we see great similarity, but to prove the equivalence in \thref{def:lascar-strong-type} no independence relation was needed. In this very general setting we use the independence relation as a replacement for the uses compactness. Proving the equivalence of the above conditions without a nice independence relation seems a lot harder, if not impossible in this generality.
\begin{remark}
\thlabel{rem:lgtp-equivalences-only-3-amalgamation}
In \thref{thm:lascar-strong-type-equivalent-formulations} we only required $\ind$ to have \textsc{3-amalgamation}. So the assumptions of the theorem do not mention anything about Lascar strong Galois types. That means that, in the presence of such an independence relation, we can take any of the equivalent conditions in \thref{thm:lascar-strong-type-equivalent-formulations} as the definition for Lascar strong Galois types, without any circularity in the definitions.
\end{remark}
The remainder of this section is devoted to proving \thref{thm:lascar-strong-type-equivalent-formulations}.
\begin{lemma}
\thlabel{lem:easy-implications-equivalence-relations}
For any $(a, b, c; M)$ and any independence relation $\ind$ we always have
\[
a \equiv^{\ind}_c b \enspace \implies \enspace
a \equiv^2_c b \enspace \implies \enspace
a \equiv^B_c b \enspace \implies \enspace
\Lgtp(a / c; M) = \Lgtp(b / c; M).
\]
\end{lemma}
\begin{proof}
The first implication follows because any strong 2-$\ind_c$-Morley sequence is in particular strongly 2-$c$-indiscernible. The final implication follows because having the same Lascar strong Galois type is a bounded invariant equivalence relation, see \thref{ex:bounded-invariant-equivalence-relations}(iv). We prove the middle implication. So let $\equiv$ be a bounded invariant equivalence relation over $C$. It is enough to prove that $a \sim^2_c b$ implies $a \equiv_c b$. Let $\kappa$ be the bound of $\equiv$. Since $a \sim^2_c b$ we find a 2-$c$-indiscernible sequence $(a_i)_{i < \kappa^+}$ in some extension $N$ of $M$ with $a$ and $b$ on it. Without loss of generality we may assume $a_0 = a$ and $a_1 = b$. By boundedness we find $i < j < \kappa^+$ such that $a_i \equiv_c a_j$. By 2-$c$-indiscernibility we have $\gtp(a, b, c; N) = \gtp(a_i, a_j, c; N)$. So $a \equiv_c b$ follows from invariance.
\end{proof}
\begin{lemma}
\thlabel{lem:finding-strong-morley-sequences}
Suppose that $\ind$ is a basic independence relation that also satisfies \textsc{3-amalgamation}. Let $m$ be an arrow with a model as domain and let $\delta \geq 2$ be any ordinal (possibly finite). Then any 2-$\ind_m$-Morley sequence $(a_i)_{i < \delta}$ is a strong 2-$\ind_m$-Morley sequence. In particular $a \ind^N_m b$ and $\gtp(a, m; N) = \gtp(b, m; N)$ implies $a \sim^{\ind}_m b$.
\end{lemma}
\begin{proof}
By \thref{rem:independence-theorem-lemma-no-lgtp} we can apply \thref{lem:independence-theorem-lemma}, the generalised independence theorem, while avoiding referring to Lascar strong Galois types. We can thus inductively apply \thref{lem:independence-theorem-lemma} to elongate $(a_i)_{i < \delta}$ to any length we want. We do this by letting $a_0$, $a_1$ and $m$ play the roles of $b$, $a$ and $c$ respectively. The final claim follows because in that case $(a, b)$ is a 2-$\ind_m$-Morley sequence of length two.
\end{proof}
\begin{lemma}
\thlabel{lem:same-type-over-model-implies-strong-morley-equivalent}
Suppose that $\ind$ is a basic independence relation that also satisfies \textsc{3-amalgamation}. If $\gtp(a, m; N) = \gtp(b, m; N)$, where $\dom(m)$ is a model, then $a \equiv^{\ind}_m b$.
\end{lemma}
\begin{proof}
By \textsc{Existence} we have $a \ind_m^N m$, so we can apply \textsc{Extension} to find $N \to N'$ and some $(a'; N')$ such that $a' \ind_m^{N'} N$ while $\gtp(a', m; N') = \gtp(a, m; N')$. By \textsc{Monotonicity} we then have $a' \ind_m^{N'} a$ and $a' \ind_m^{N'} b$. So by \thref{lem:finding-strong-morley-sequences} we have $a \sim^{\ind}_m a' \sim^{\ind}_m b$, and we are done.
\end{proof}
\begin{proof}[Proof of \thref{thm:lascar-strong-type-equivalent-formulations}]
By \thref{lem:easy-implications-equivalence-relations} we only need to prove that $\Lgtp(a/c; M) = \Lgtp(b/c; M)$ implies $a \equiv^2_c b$. It is enough to prove that $(a/c; M) \sim_{\Lgtp} (b/c; M)$ implies $a \equiv^2_c b$. So let $M \to N$ be an extension and let $m_0: M_0 \to N$ be such that $c$ factors through $m_0$, $M_0$ is a model and $\gtp(a, m_0; N) = \gtp(b, m_0; N)$. Then by \thref{lem:same-type-over-model-implies-strong-morley-equivalent} we get $a \equiv^{\ind}_{m_0} b$. So by \thref{lem:easy-implications-equivalence-relations} we have $a \equiv^2_{m_0} b$. Any strongly 2-$m_0$-indiscernible sequence is also strongly 2-$c$-indiscernible, because $c$ factors through $m_0$. So we conclude that indeed $a \equiv^2_c b$.
\end{proof}
\begin{remark}
\thlabel{rem:lgtp-also-equivalent-to-2-morley-sequences}
The proof of \thref{thm:lascar-strong-type-equivalent-formulations} also shows that if $\dom(c)$ is a model then $\Lgtp(a/c; M) = \Lgtp(b/c; M)$ is further equivalent to $a \equiv^{\ind}_c b$.
\end{remark}
\begin{remark}
\thlabel{rem:lgtp-equivalences-tuples}
We have stated \thref{thm:lascar-strong-type-equivalent-formulations} for single arrows, rather than for tuples of arrows. We briefly sketch how we can get the result for tuples of arrows as well. That is, if we replace $a$, $b$ and $c$ by $(a_i)_{i \in I}$, $(b_i)_{i \in I}$ and $(c_j)_{j \in J}$ respectively.

First we extend the definitions of (strongly) 2-indiscernible, $\equiv^2$ and $\equiv^B$ to tuples of arrows in a straightforward way. Then following the same proof as in \thref{lem:easy-implications-equivalence-relations} we get:
\begin{align*}
&(a_i)_{i \in I} \equiv^2_{(c_j)_{j \in J}} (b_i)_{i \in I} &\implies \\
&(a_i)_{i \in I} \equiv^B_{(c_j)_{j \in J}} (b_i)_{i \in I} &\implies \\
&\Lgtp((a_i)_{i \in I} / (c_j)_{j \in J}; M) = \Lgtp((b_i)_{i \in I} / (c_j)_{j \in J}; M).
\end{align*}
So we are left to prove that $\Lgtp((a_i)_{i \in I} / (c_j)_{j \in J}; M) = \Lgtp((b_i)_{i \in I} / (c_j)_{j \in J}; M)$ implies $(a_i)_{i \in I} \equiv^2_{(c_j)_{j \in J}} (b_i)_{i \in I}$. It is enough to prove that $((a_i)_{i \in I} / (c_j)_{j \in J}; M) \sim_{\Lgtp} ((b_i)_{i \in I} / (c_j)_{j \in J}; M)$ implies $(a_i)_{i \in I} \equiv^2_{(c_j)_{j \in J}} (b_i)_{i \in I}$. So let $M \to N$ be an extension and let $m_0: M_0 \to N$ be such that all arrows in $(c_j)_{j \in J}$ factor through $m_0$, $M_0$ is a model and $\gtp((a_i)_{i \in I}, m_0; N) = \gtp((b_i)_{i \in I}, m_0; N)$. Pick some $d: D \to N$ such that every arrow in $(a_i)_{i \in I}$ factors through $d$. We then find an extension $N \to N'$ and $d': D \to N'$ such that every arrow in $(b_i)_{i \in I}$ factors through $d'$ and $\gtp(d, m_0; N') = \gtp(d', m_0; N')$. Now we can apply the original result \thref{thm:lascar-strong-type-equivalent-formulations} to obtain $d \equiv^2_{m_0} d'$ and hence $(a_i)_{i \in I} \equiv^2_{(c_j)_{j \in J}} (b_i)_{i \in I}$, as required.
\end{remark}
\begin{remark}
\thlabel{rem:lgtp-equivalences-given-independence-theorem}
We only assumed \textsc{3-amalgamation}. If we also assume \textsc{Independence Theorem} together with $\base(\ind) = \C$ we get a little bit more, namely that $\Lgtp(a/c; M) = \Lgtp(b/c; M)$ is further equivalent to $a \equiv^{\ind}_c b$. This happens for example in any simple thick positive theory where $\ind$ is the usual dividing independence (see \thref{ex:simple-canonicity-positive-logic}).

The proof of this is largely the same as the proof in this section. We sketch where some changes would need to be made. We adjust \thref{lem:finding-strong-morley-sequences} as follows: any 2-$\ind_c$-Morley sequence $(a_i)_{i < \delta}$ such that $\Lgtp(a_i / c; M) = \Lgtp(a_j / c; M)$ for all $i < j < \delta$ is a strong 2-$\ind_c$-Morley sequence. Here the extra assumption ``$\Lgtp(a_i / c; M) = \Lgtp(a_j / c; M)$'' is necessary to still apply \thref{lem:independence-theorem-lemma}, and so the proof goes through. Then \thref{lem:same-type-over-model-implies-strong-morley-equivalent} can be restated as $\Lgtp(a/c; M) = \Lgtp(b/c; M)$ implies $a \equiv^{\ind}_c b$. The only change in the proof is that we need to apply strong extension, \thref{cor:strong-extension-dual}. This then already concludes the proof.

This also tells us that in this case we will need at most two strong 2-$\ind_c$-Morley sequences to connect $a$ and $b$, whenever they have the same Lascar strong Galois type. That is, there is some $a'$ in an extension of $M$ such that $a \sim^{\ind}_c a' \sim^{\ind}_c b$. In particular this also means that we need at most two strongly 2-$c$-indiscernible sequences to connect $a$ and $b$, because strong 2-$\ind_c$-Morley sequences are in particular strongly 2-$c$-indiscernible.
\end{remark}

\bibliographystyle{alpha}
\bibliography{bibfile}


\end{document}